\documentclass[11pt,reqno,twoside]{article}
\usepackage{amsthm}

\usepackage[hang]{footmisc}
\usepackage{lipsum}

\usepackage{cmap} 

\usepackage[T1]{fontenc}
\usepackage[utf8]{inputenc}
\usepackage{graphicx}
\usepackage{array}
\usepackage{placeins}
\usepackage{enumerate}
\usepackage{booktabs}
\usepackage{algcompatible}
\usepackage[most]{tcolorbox}
\usepackage{mdframed}

\usepackage{verbatim}
\newcommand{\comments}[1]{}

\usepackage{soul}

\usepackage{setspace}

\let\counterwithin\relax  
\usepackage{lmodern} 

\usepackage{comment}

\usepackage{bm} 

\usepackage{bbold}

\usepackage{amsmath,amsbsy,amsgen,amscd,amsthm,amsfonts,amssymb}

\usepackage{float}       

\usepackage[centering,top=1.1in,bottom=1.4in,left=1in,right=1in]{geometry}

\usepackage[sf,bf,compact]{titlesec}
\usepackage[mathscr]{euscript}

\usepackage{enumitem}
\usepackage[dvipsnames]{xcolor}

\definecolor{dark-gray}{gray}{0.3}
\definecolor{dkgray}{rgb}{.4,.4,.4}
\definecolor{dkblue}{rgb}{0,0,.5}
\definecolor{medblue}{rgb}{0,0,.75}
\definecolor{rust}{rgb}{0.5,0.1,0.1}

\usepackage{url}
\usepackage[colorlinks=true]{hyperref}
\hypersetup{linkcolor=dkblue}    
\hypersetup{citecolor=rust}      
\hypersetup{urlcolor=rust}     

\usepackage[final]{microtype} 

\newtheoremstyle{myThm} 
    {\topsep}                    
    {\topsep}                    
    {\itshape}                   
    {}                           
    {\sffamily\bfseries}                   
    {.}                          
    {.5em}                       
    {}  

\newtheoremstyle{myRem} 
    {\topsep}                    
    {\topsep}                    
    {}                   
    {}                           
    {\sffamily}                   
    {.}                          
    {.5em}                       
    {}  

\newtheoremstyle{myDef} 
    {\topsep}                    
    {\topsep}                    
    {}                   
    {}                           
    {\sffamily\bfseries}                   
    {.}                          
    {.5em}                       
    {}  

\theoremstyle{myThm}
\newtheorem{theorem}{Theorem}[section]
\newtheorem{lemma}[theorem]{Lemma}
\newtheorem{proposition}[theorem]{Proposition}

\newtheorem{definition}[theorem]{Definition}

  \newtheorem{examplex}[theorem]{Example}
 \newenvironment{example}
  {\pushQED{\qed}\examplex}
  {\popQED\endexamplex}

\theoremstyle{myRem}
\newtheorem{remarkx}[theorem]{Remark}
 \newenvironment{remark}
  {\pushQED{\qed}\remarkx}
  {\popQED\endremarkx}
\usepackage{fancyhdr}
\let\originalleft\left
\let\originalright\right
\renewcommand{\left}{\mathopen{}\mathclose\bgroup\originalleft}
\renewcommand{\right}{\aftergroup\egroup\originalright}

\usepackage{mathtools}
\mathtoolsset{centercolon}  

\newcommand{\circnum}[1]{\raisebox{.5pt}{\textcircled{\raisebox{-.2pt}{\scriptsize #1}}}}

\providecommand{\mathbbm}{\mathbb} 

\newcommand{\R}{\mathbbm{R}}
\newcommand{\E}{\mathbbm{E}}

\newcommand{\mcN}{\mathcal{N}}

\newcommand{\tr}{\mathrm{Tr}}   

\newcommand{\supp}{\mathrm{supp}}
\newcommand{\Cov}{\mathrm{Cov}}

\definecolor{mygreen}{rgb}{0.13,0.55,0.13}

\usepackage[font = small, margin=30pt]{caption}

\usepackage[]{algorithm}
\usepackage{enumerate}

\usepackage{graphicx}

\usepackage{authblk}
\usepackage{chngcntr}
\usepackage{mathrsfs} 
\counterwithin{table}{section}

\usepackage[normalem]{ulem}

\usepackage{tikz}
\usetikzlibrary{shapes,arrows,arrows.meta, decorations.pathmorphing,backgrounds,positioning,fit,petri}
\usetikzlibrary{calc}
\usetikzlibrary{matrix}
\usetikzlibrary{backgrounds}
\usetikzlibrary{shapes.geometric}
\usetikzlibrary{patterns}
\usepackage{pgfplots}
\pgfplotsset{compat=newest}
\usepgfplotslibrary{groupplots}
\pgfdeclarelayer{background}
\pgfdeclarelayer{bBack}
\pgfsetlayers{bBack,background,main}
\usetikzlibrary{fit}

\usepackage{subcaption}

\title{Sharp Frobenius-Norm Concentration for Sample Moment Tensors}

\author{Jiaheng Chen and Daniel Sanz-Alonso}

\date{University of Chicago}

\vspace{.5in}

\makeatletter\@addtoreset{section}{part}\makeatother%
\numberwithin{equation}{section}

\newcommand{\upperRomannumeral}[1]{\uppercase\expandafter{\romannumeral#1}}

\makeatletter\@addtoreset{section}{part}\makeatother%
\numberwithin{equation}{section}

\begin{document}
\maketitle 


\renewcommand{\thefootnote}{\fnsymbol{footnote}}


\vspace{-1.7em}
\abstract{
 This paper establishes sharp dimension-free Frobenius-norm concentration inequalities for sample moment tensors. Our bounds are optimal over the sub-Gaussian class, while for Gaussian data we obtain matching two-sided estimates. We also identify a parity effect: the intermediate Gaussian and sub-Gaussian scales coincide at odd tensor orders but differ at even orders. Our second main result, which supplies the weak-moment estimate behind these bounds, proves a  dimension-free moment bound for
 Hilbert-valued  polynomials under Gaussian convex domination. The proof of this bound combines a recently developed structured martingale coupling theorem with new estimates for moment tensors of uniformly log-concave distributions.  
}

\bigskip


\section{Introduction}

Moment tensors are fundamental objects in probability and statistics. This
paper establishes sharp dimension-free Frobenius-norm concentration
inequalities for their empirical counterparts. Let
\(X_1,\ldots,X_N\) be independent copies of a centered random vector
\(X\in\mathbb R^d\) with covariance matrix \(\Sigma\). For an integer
\(p\ge2\), we study, in the sub-Gaussian and Gaussian settings, the
deviation
\[
\bigg\|
\frac1N\sum_{i=1}^N X_i^{\otimes p}
-
\E X^{\otimes p}
\bigg\|_{\mathsf F},
\]
where \(X^{\otimes p}:=X\otimes\cdots\otimes X\) denotes the \(p\)-fold
tensor product of \(X\), and \(\|\cdot\|_{\mathsf F}\) denotes the
Frobenius norm.

Our first main result, Theorem~\ref{thm:main1}, determines the sharp
non-asymptotic dependence of this deviation on the sample size, the
confidence level, and intrinsic covariance quantities. Our upper bound is
sharp over the class of sub-Gaussian distributions, while for Gaussian data we obtain
matching two-sided estimates for every covariance matrix $\Sigma$. The results also
identify a quantitative parity effect: the Gaussian and sub-Gaussian
intermediate covariance scales coincide at odd tensor orders, whereas at
even orders the sub-Gaussian scale can exceed the Gaussian one by a factor
as large as \(\sqrt{\operatorname{rank}(\Sigma)}\).

The main technical difficulty is the weak-moment estimate for the centered
tensor power of a sub-Gaussian random vector, a homogeneous polynomial in
that vector. Recent advances in the theory of sub-Gaussian distributions
\cite{vanhandel2025subgaussian} show that sub-Gaussian random vectors are
dominated in convex order by Gaussian ones, but convex domination does not
directly control the moments of such polynomials. Our second main result,
Theorem~\ref{thm:main2}, overcomes this obstruction by establishing a
dimension-free moment bound for Hilbert-valued polynomials under Gaussian
convex domination. Its proof uses a
structured martingale coupling recently developed in
\cite{hua2026talagrand} to expand the target polynomial into a conditional
Gaussian term plus lower-degree correction terms; the corrections are
controlled by new moment-tensor bounds for uniformly log-concave
distributions, and the argument closes by induction on the degree of the
polynomial.

\paragraph{Notation}\label{sec:notation}

We write \(a\lesssim_p b\) if \(a\le C_p b\) for a constant
\(C_p>0\) that depends only on \(p\), and write
\(a\asymp_p b\) if both \(a\lesssim_p b\) and
\(b\lesssim_p a\). We omit the subscript when the implicit constants
are universal. We also write \(a\lor b:=\max\{a,b\}\) and
\(a\land b:=\min\{a,b\}\). For a vector \(x\in\mathbb R^d\), we denote its Euclidean norm by
\(\|x\|_2\). For order-\(p\) tensors
\(A,B\in\mathbb R^{d_1\times\cdots\times d_p}\), we define the
Frobenius inner product and norm by $\langle A,B\rangle_{\mathsf F}
:=
\sum_{i_1,\ldots,i_p}
A_{i_1,\ldots,i_p}B_{i_1,\ldots,i_p}$ and $
\|A\|_{\mathsf F}
:=
\langle A,A\rangle_{\mathsf F}^{1/2}$.
The operator norm of \(A\) is defined by $\|A\|
:=
\sup_{\|u_1\|_2,\ldots,\|u_p\|_2\le1}
\left|
\left\langle
A,u_1\otimes\cdots\otimes u_p
\right\rangle_{\mathsf F}
\right|$. When \(p=2\), these norms coincide with the usual Frobenius and
operator norms of a matrix.

\section{Main results}\label{sec:main}

Here we state and discuss the two main theorems of the paper. Subsection~\ref{subsec:main1} presents Theorem~\ref{thm:main1} on the concentration of sample moment tensors, while Subsection~\ref{subsec:main2} presents Theorem~\ref{thm:main2}, a polynomial moment bound under Gaussian convex domination. The latter is a key ingredient in the proof of the former. The two results are proved in Sections~\ref{sec:proof1} and~\ref{sec:proof2}, respectively.

\subsection{Concentration of sample moment tensors}\label{subsec:main1}

The Orlicz \(\psi_2\) norm and the
\(L_q\) norm of a real-valued random variable \(Z\) are defined by  
\[
\|Z\|_{\psi_2}
:=
\inf\left\{
t>0:\mathbb{E}\exp(Z^2/t^2)\le 2
\right\},
\qquad
\|Z\|_{L_q} := \big(\mathbb{E}|Z|^q\big)^{1/q}.
\]
We say that a random vector \(X \in \R^d\)  is sub-Gaussian if there exists \(K>0\) such
that
\begin{align}\label{eq:def-subg}
\|\langle X,v\rangle\|_{\psi_2}
\le K \|\langle X,v\rangle\|_{L_2},
\qquad
\text{for any } v\in \mathbb{S}^{d-1},
\end{align}
where $\mathbb{S}^{d-1}$ denotes the unit sphere in $\R^d.$ 
The smallest such constant $K>0$ is called the sub-Gaussian constant of $X$. We use $\mcN(0,\Sigma)$ to denote the Gaussian distribution with mean zero and covariance matrix $\Sigma$. Throughout the paper, all covariance matrices are allowed to be singular, but are assumed to be nonzero. 

Our first main result establishes sharp Frobenius-norm concentration bounds for sample moment tensors of sub-Gaussian and Gaussian random vectors.

\begin{theorem}[Concentration of sample moment tensors]\label{thm:main1}
Let \(X_1,\ldots,X_N\) be i.i.d. copies of a centered sub-Gaussian random
vector \(X\in\mathbb R^d\) with covariance matrix \(\Sigma\) and
sub-Gaussian constant \(K\). Let \(p\ge 2\) be an integer. Then, for every
\(q\ge1\),
\[
\left(
\E
\bigg\|
\frac1N\sum_{i=1}^N X_i^{\otimes p}
-
\E X^{\otimes p}
\bigg\|_{\mathsf F}^q
\right)^{1/q}
\lesssim_p
K^p
\left(
\frac{\tr(\Sigma)^{p/2}}{\sqrt N}
+
\rho_p(\Sigma)\sqrt{\frac{q\land N}{N}}
+
\|\Sigma\|^{p/2}\frac{q^{p/2}}{N}
\right),
\]
where
\[
\rho_p(\Sigma)
:=
\begin{cases}
\|\Sigma\|_{\mathsf F}^{p/2},
& p \text{ even},\\[3pt]
\|\Sigma\|_{\mathsf F}^{(p-1)/2}\|\Sigma\|^{1/2},
& p \text{ odd}.
\end{cases}
\]
Moreover, if $X$ is Gaussian, then
\[
\left(\E \bigg\|
\frac{1}{N}
\sum_{i=1}^N
X_i^{\otimes p}-\E X^{\otimes p}
\bigg\|_{\mathsf{F}}^q\right)^{1/q}
\asymp_p
\frac{\tr(\Sigma)^{p/2}}{\sqrt{N}}
+
\rho_{p,G}(\Sigma)\sqrt{\frac{q}{N}}
+
\|\Sigma\|^{p/2}\frac{q^{p/2}}{N},
\]
where
\[
\rho_{p,G}(\Sigma)
:=
\begin{cases}
\|\Sigma\|_{\mathsf{F}}^{p/2-1}\|\Sigma\|, & p \text{ even},\\[3pt]
\|\Sigma\|_{\mathsf{F}}^{(p-1)/2}\|\Sigma\|^{1/2}, & p \text{ odd}.
\end{cases}
\]
\end{theorem}

We make a few remarks on Theorem~\ref{thm:main1}.

\begin{remark}\leavevmode
\begin{enumerate}
\item
The first upper bound in Theorem~\ref{thm:main1} is sharp over the class
of sub-Gaussian distributions. More precisely, in
Example~\ref{example:subgaussian_tensor_sharpness} we construct a centered
random vector \(X\) with covariance \(\Sigma\) and sub-Gaussian constant
\(K\asymp1\) such that, for every integer \(p\ge2\), every \(q\ge1\), and
i.i.d. copies \(X_1,\ldots,X_N\) of \(X\),
\[
\left(
\E \bigg\|
\frac1N\sum_{i=1}^N X_i^{\otimes p}-\E X^{\otimes p}
\bigg\|_{\mathsf F}^q
\right)^{1/q}
\gtrsim_p
\frac{\tr(\Sigma)^{p/2}}{\sqrt N}
+
\rho_p(\Sigma)\sqrt{\frac{q\land N}{N}}
+
\|\Sigma\|^{p/2}\frac{q^{p/2}}{N},
\]
where \(\rho_p(\Sigma)\) is defined in
Theorem~\ref{thm:main1}. Hence all three terms in the upper bound are
necessary, up to constants depending only on \(p\).

\item 
It is instructive to compare the sub-Gaussian and Gaussian bounds. The trace
and large-deviation scales coincide in the two settings. When \(p\) is odd,
the intermediate scales also coincide:
\[
\rho_p(\Sigma)
=
\rho_{p,G}(\Sigma)
=
\|\Sigma\|_{\mathsf F}^{(p-1)/2}\|\Sigma\|^{1/2}.
\]
By contrast, when \(p\) is even,
\[
\rho_p(\Sigma)=\|\Sigma\|_{\mathsf F}^{p/2},
\qquad
\rho_{p,G}(\Sigma)
=
\|\Sigma\|_{\mathsf F}^{p/2-1}\|\Sigma\|.
\]
Thus, at even orders, passing from the Gaussian distribution to the full
sub-Gaussian class may enlarge the covariance-dependent coefficient in the intermediate scale by the factor
$
\|\Sigma\|_{\mathsf F} /\|\Sigma\|,
$
which can be as large as
\(\sqrt{\operatorname{rank}(\Sigma)}\). This enlargement is intrinsic, as illustrated by
Example~\ref{example:subgaussian_tensor_sharpness}; the parity effect is discussed further in Subsection \ref{subsec:sharpness}. 

\item Taking \(q=u\) in Theorem~\ref{thm:main1} and applying Markov's inequality
yields that, for every \(u\ge1\), with probability at least \(1-e^{-u}\),
\[
\bigg\|
\frac1N\sum_{i=1}^N X_i^{\otimes p}-\E X^{\otimes p}
\bigg\|_{\mathsf F}
\lesssim_p
K^p
\left(
\frac{\tr(\Sigma)^{p/2}}{\sqrt N}
+
\rho_p(\Sigma)\sqrt{\frac{u\land N}{N}}
+
\|\Sigma\|^{p/2}\frac{u^{p/2}}N
\right).
\]
If \(X\) is Gaussian,  with probability at least \(1-e^{-u}\),
\[
\bigg\|
\frac1N\sum_{i=1}^N X_i^{\otimes p}-\E X^{\otimes p}
\bigg\|_{\mathsf F}
\lesssim_p
\frac{\tr(\Sigma)^{p/2}}{\sqrt N}
+
\rho_{p,G}(\Sigma)\sqrt{\frac{u}{N}}
+
\|\Sigma\|^{p/2}\frac{u^{p/2}}{N}.
\]
\end{enumerate}
\end{remark}

\paragraph{Related work}
\emph{The covariance case \(p=2\).}
For \(p=2\), the problem reduces to sample covariance concentration. In operator norm, \cite[Theorems~4 and~9]{koltchinskii2017concentration} established sharp
dimension-free bounds governed by the effective rank
$r(\Sigma):=\tr(\Sigma)/\|\Sigma\|$ using empirical-process methods. For the corresponding Gaussian bound, see also
\cite[Section~5]{van2017structured} for an elementary alternative proof based
on decoupling and comparison inequalities. The optimal
leading constants were obtained in
\cite[Theorem~2.3]{han2022exact} via the Gaussian min--max theorem. A general-rank analogue for
empirical averages of positive semi-definite random matrices was obtained in
\cite[Theorem~1]{zhivotovskiy2024dimension} using PAC-Bayesian methods. Extensions to sample
cross-covariance concentration were obtained in
\cite{chen2026concentration}.

For simplicity, throughout the following Frobenius-norm comparison, we
assume that the sub-Gaussian constant \(K\) in \eqref{eq:def-subg} is
bounded by a universal constant, since our focus is on the dependence on
\(N\), \(u\), and intrinsic covariance quantities. Under a condition equivalent
to \eqref{eq:def-subg}, Bunea and Xiao
\cite[Proposition~A.3]{bunea2015sample} proved that, for \(u\ge1\), with
probability at least \(1-e^{-u}\),
\[
\|\widehat\Sigma_N-\Sigma\|_{\mathsf F}
\lesssim
 \tr(\Sigma)
\left(
\sqrt{\frac{u}{N}}
+
\frac{u}{N}
\right),
\qquad
\widehat\Sigma_N
:=
\frac{1}{N}\sum_{i=1}^{N}X_i\otimes X_i.
\]
In particular,
\(\E\|\widehat\Sigma_N-\Sigma\|_{\mathsf F}
\lesssim\tr(\Sigma)/\sqrt{N}\). Their deviation bound, however, uses \(\tr(\Sigma)\) in both the
moderate- and large-deviation regimes. More recently, Puchkin, Noskov, and Spokoiny
\cite[Theorems~2.4 and~2.6]{puchkin2025sharper} studied the finer problem of concentration of the squared Frobenius error around
its expectation. Their upper-tail result assumes a uniform fourth-moment
bound for whitened quadratic forms under quadratic exponential tilts;
see \cite[Assumption~2.1]{puchkin2025sharper}. This condition is not implied, with dimension-free constants, by the
sub-Gaussianity condition \eqref{eq:def-subg}; see
\cite[Proposition~2.3]{puchkin2025sharper}. Treating the regularity
parameters in \cite[Assumption~2.1]{puchkin2025sharper} as fixed, their upper- and lower-tail bounds
imply that, for every \(u\ge1\), with probability at least \(1-e^{-u}\),
\[
\left|
\|\widehat\Sigma_N-\Sigma\|_{\mathsf F}^{2}
-
\E\|\widehat\Sigma_N-\Sigma\|_{\mathsf F}^{2}
\right|
\lesssim
\frac{
\|\Sigma\|_{\mathsf F}^{2}\sqrt{u}
+
\|\Sigma\|^{2}u
}{N},
\]
provided that $N
\gtrsim
r(\Sigma)^6+u^2r(\Sigma)^2$.
Thus, at fixed confidence, the centered squared loss is controlled at
the scale \(\|\Sigma\|_{\mathsf F}^{2}/N\), which can be much smaller
than the typical squared-error scale \(\tr(\Sigma)^2/N\). The results of
\cite{puchkin2025sharper} therefore provide finer concentration of the
centered loss on a subclass of sub-Gaussian distributions.

Compared with \cite{bunea2015sample} and
\cite{puchkin2025sharper}, our \(p=2\) result in
Theorem~\ref{thm:main1} characterizes, up to constant factors, the sharp
\(L_q\)-behavior of the Frobenius error over the full sub-Gaussian class
satisfying \eqref{eq:def-subg}, for all \(N\) and \(q\), without any
sample-size restriction. In particular, for every \(u\ge1\), with
probability at least \(1-e^{-u}\),
\[
\|\widehat\Sigma_N-\Sigma\|_{\mathsf F}
\lesssim
\frac{\tr(\Sigma)}{\sqrt N}
+
\|\Sigma\|_{\mathsf F}
\sqrt{\frac{u\land N}{N}}
+
\|\Sigma\|\frac{u}{N}.
\]
The matching lower bounds are provided by
Example~\ref{example:subgaussian_tensor_sharpness}.

\emph{Higher-order moment tensors \(p>2\).}
Previous work on higher-order sample moment tensors has focused primarily
on the operator norm. Sharp dimension-free concentration
bounds for sub-Gaussian sample moment tensors were obtained in
\cite{al2025sharp}, with a simpler proof subsequently given in
\cite{abdalla2026dimension}. The asymmetric and cross-moment tensor settings were
treated in \cite{chen2026sharp}, with a possible logarithmic factor in
certain effective-rank regimes. The present work instead develops sharp dimension-free Frobenius-norm concentration bounds for higher-order sample moment tensors. More precisely, for every order \(p\ge2\), we obtain sharp \(L_q\)-bounds,
and hence high-probability bounds, prove their optimality over the
sub-Gaussian class through a family of examples, derive matching two-sided
Gaussian estimates, and identify a parity-dependent gap between the
sub-Gaussian and Gaussian settings.

\smallskip

\paragraph{Proof idea of Theorem~\ref{thm:main1}} The proof of Theorem~\ref{thm:main1} is given in
Section~\ref{sec:proof1}. We briefly describe its main steps here to explain
the role of the polynomial moment bound developed in the next subsection.
For \(1\le i\le N\), set
\(Z_i:=X_i^{\otimes p}-\E X^{\otimes p}\).
For \(q\ge2\), Lemma~\ref{lem:Hilbert-moment-reduction} gives
\begin{align*}
\left(
\E
\bigg\|
\frac1N\sum_{i=1}^N Z_i
\bigg\|_{\mathsf F}^q
\right)^{1/q}
\asymp{}&
\frac1N
\left\|
\bigg(
\sum_{i=1}^N\|Z_i\|_{\mathsf F}^2
\bigg)^{1/2}
\right\|_{L_q}
+
\frac1N
\sup_{\|A\|_{\mathsf F}\le1}
\left\|
\sum_{i=1}^N
\langle A,Z_i\rangle_{\mathsf F}
\right\|_{L_{q/2}}.
\end{align*}
This reduces the strong moments of the tensor-valued sum to a
square-function moment and the weak moments of its scalar projections. The
square-function term is controlled using moment bounds for \(\|X\|_2\) and
Lata{\l}a's estimate for sums of nonnegative random variables
\cite{latala1997estimation}. For the weak-moment term, Lata{\l}a's estimate
for centered sums further reduces the problem to estimating the weak moments
of a single centered tensor power:
\begin{equation}
\label{eq:weak-moment-target}
\sup_{\|A\|_{\mathsf F}\le1}
\left\|
\left\langle
A,X^{\otimes p}-\E X^{\otimes p}
\right\rangle_{\mathsf F}
\right\|_{L_q},\qquad q\ge 1.
\end{equation}
Obtaining the required bound for a general sub-Gaussian vector \(X\) is the
main challenge. This is precisely where Theorem~\ref{thm:main2}, the
polynomial moment bound under Gaussian convex domination developed in the
next subsection, enters the argument.

\subsection{Polynomial moments under Gaussian convex domination}
\label{subsec:main2}

The discussion in the previous subsection reduces the proof of
Theorem~\ref{thm:main1} to controlling the weak moment in \eqref{eq:weak-moment-target}. For \(A\in(\mathbb R^d)^{\otimes p}\), define 
\begin{align}\label{eq:T-A}
T_A(x_1,\ldots,x_p)
:=
\left\langle
A,x_1\otimes\cdots\otimes x_p
\right\rangle_{\mathsf F}.
\end{align}
Since
\(\langle A,X^{\otimes p}\rangle_{\mathsf F}
=T_A(X,\ldots,X)\),
it suffices to establish a dimension-free moment bound
for the polynomial \(T_A(X,\ldots,X)\) when \(X\)
is sub-Gaussian. To obtain sharp bounds, we build on recent developments in the theory of
sub-Gaussian random vectors
\cite{vanhandel2025subgaussian,hua2026talagrand}.

Let \(U\) and \(V\) be random vectors in \(\mathbb R^d\). We say that
\(U\) is dominated by \(V\) in convex order, and write
\(U\preceq_{\mathrm{cx}}V\), if $\E\varphi(U)\le\E\varphi(V)$ for every convex function
\(\varphi:\mathbb R^d\to\mathbb R\).
A recent result of van Handel
\cite[Theorem~1.1]{vanhandel2025subgaussian} shows that sub-Gaussian
random vectors are dominated in convex order by Gaussian random vectors.
More precisely, in the setting of \eqref{eq:def-subg}, 
\begin{align}\label{eq:convex-dominatation}
(CK)^{-1}X
\preceq_{\mathrm{cx}}
G_\Sigma,
\end{align}
where $C$ is a universal constant and
\(G_\Sigma\sim\mcN(0,\Sigma)\). This comparison does not directly give the required polynomial moment
bound. Indeed, when \(p\ge2\), the function $z\longmapsto
|T_A(z,\ldots,z)|^q$ is not convex in general. To tackle this nonconvex polynomial estimate for sub-Gaussian random
vectors, we develop a new polynomial moment bound under Gaussian convex
domination in Theorem~\ref{thm:main2}. 

We first
introduce some notation. For a \(p\)-linear map $T:(\mathbb R^d)^p\to\mathcal H$,
where \(\mathcal H\) is a finite-dimensional Hilbert space, we define
\[
\|T\|_{\mathsf F}^2
:=
\sum_{i_1,\ldots,i_p=1}^d
\big\|
T(e_{i_1},\ldots,e_{i_p})
\big\|_{\mathcal H}^2.
\]
Here \(e_1,\ldots,e_d\) is any orthonormal basis of \(\mathbb R^d\);
the definition does not depend on this choice.
We also set
\begin{align}\label{eq:Psi_p}
\Psi_p(q,\Sigma)
:=
\sum_{\substack{0\le j\le p\\j\equiv p\;(\mathrm{mod}\,2)}}
\|\Sigma\|_{\mathsf F}^{(p-j)/2}
\|\Sigma\|^{j/2}q^{j/2},\qquad \Psi_p^\circ(q,\Sigma)
:=
\sum_{\substack{1\le j\le p\\j\equiv p\;(\mathrm{mod}\,2)}}
\|\Sigma\|_{\mathsf F}^{(p-j)/2}
\|\Sigma\|^{j/2}q^{j/2}.
\end{align}
As will become clear later, the summand indexed by \(j\) corresponds to a
Gaussian chaos component of order \(j\); the two scales differ only in the
term \(j=0\), which is present only at even orders.

\begin{theorem}[Polynomial moments under Gaussian convex domination]
\label{thm:main2}
Let \(G_\Sigma\sim\mcN(0,\Sigma)\), where
\(\Sigma\succeq0\), and let \(Z\in\mathbb R^d\) be a centered random
vector such that \(Z\preceq_{\mathrm{cx}}G_\Sigma\). Then, for every
integer \(p\ge1\), every \(q\ge1\), every finite-dimensional Hilbert
space \(\mathcal H\), and every \(p\)-linear map
\[
T:(\mathbb R^d)^p\to\mathcal H,
\]
one has
\[
\left(
\E
\big\|
T(Z,\ldots,Z)
\big\|_{\mathcal H}^q
\right)^{1/q}
\lesssim_p
\Psi_p(q,\Sigma)\|T\|_{\mathsf F}.
\]
\end{theorem}

In particular, by \eqref{eq:convex-dominatation},
\(Z:=(CK)^{-1}X\) satisfies
\(Z\preceq_{\mathrm{cx}}G_\Sigma\) if $X$ is sub-Gaussian with sub-Gaussian constant $K$. Applying
Theorem~\ref{thm:main2} with \(\mathcal H=\mathbb R\) and \(T=T_A\) in \eqref{eq:T-A},
and using
\(\|T_A\|_{\mathsf F}=\|A\|_{\mathsf F}\) together with
\(\|V-\E V\|_{L_q}\le2\|V\|_{L_q}\), yields
\begin{align*}
\sup_{\|A\|_{\mathsf F}\le1}
\left\|
\left\langle
A,X^{\otimes p}-\E X^{\otimes p}
\right\rangle_{\mathsf F}
\right\|_{L_q}
&\lesssim_p
K^p\Psi_p(q,\Sigma) \\
&\asymp_p
K^p
\begin{cases}
\|\Sigma\|_{\mathsf F}^{p/2}
+
\|\Sigma\|^{p/2}q^{p/2},
&p\text{ even},\\[4pt]
\|\Sigma\|_{\mathsf F}^{(p-1)/2}
\|\Sigma\|^{1/2}\sqrt q
+
\|\Sigma\|^{p/2}q^{p/2},
&p\text{ odd}.
\end{cases}
\end{align*}
The final comparison follows because the terms defining
\(\Psi_p(q,\Sigma)\) form a geometric progression. This is precisely the
sub-Gaussian weak tensor moment estimate stated in
Lemma~\ref{lem:subgaussian-directional-tensor-moments}, and the point
at which Theorem~\ref{thm:main2} enters the proof of
Theorem~\ref{thm:main1}. Although the application of Theorem~\ref{thm:main2} here is scalar-valued, our proof strategy directly yields the Hilbert-valued result. Indeed,
the argument proceeds by induction on \(p\), and the induction step requires
applying the induction hypothesis to multilinear maps taking values in
tensor-product Hilbert spaces.

\begin{remark}
\label{rem:main2-sharpness}
Theorem~\ref{thm:main2} is sharp for every \(\Sigma\succeq0\) and every
\(q\ge1\). Indeed, since
\(G_\Sigma\preceq_{\mathrm{cx}}G_\Sigma\), the choice
\(Z=G_\Sigma\) is admissible. Taking \(\mathcal H=\mathbb R\) and
\(T=T_A\) as in \eqref{eq:T-A}, the noncentered Gaussian estimate in
Lemma~\ref{lem:subgaussian-directional-tensor-moments} gives
\[
\sup_{\|A\|_{\mathsf F}=1}
\left\|
\left\langle A,G_\Sigma^{\otimes p}\right\rangle_{\mathsf F}
\right\|_{L_q}
\asymp_p
\Psi_p(q,\Sigma).
\]
Thus, the martingale coupling argument used in the proof of
Theorem~\ref{thm:main2}, outlined below and carried out in
Section~\ref{sec:proof2}, recovers the sharp Gaussian scale under convex
domination.
\end{remark}

\smallskip

\paragraph{Proof idea of Theorem~\ref{thm:main2}}
The proof of Theorem \ref{thm:main2} is given in Section~\ref{sec:proof2}. Here we provide a high-level overview.

 Our starting point  
is a structural strengthening of Strassen's martingale
coupling theorem obtained by Hua, Song, and Tudose in their recent
resolution of Talagrand's convexity conjecture
\cite{hua2026talagrand}. A classical theorem of Strassen
\cite{strassen1965existence} gives an equivalent characterization of convex domination:
\(U\preceq_{\mathrm{cx}}V\) if and only if \(U\) and \(V\) admit a
martingale coupling, that is, a coupling satisfying $\E[V\mid U]=U$. Strassen's theorem is an abstract existence result and gives no further
information about the conditional laws of the coupling. In contrast,
\cite[Proposition~3.3]{hua2026talagrand} states the
following stronger structural result. Let \(U\) be a centered, finitely supported random vector, and
let \(V\) be a centered random vector with an absolutely continuous law.
If $U\preceq_{\mathrm{cx}}(1-\varepsilon)V$ for some $\varepsilon\in (0,1)$, then \(U\) and \(V\) admit a martingale coupling such that, for every
\(u\in\supp(U)\), the conditional law of \(V\) given \(U=u\)
has a log-concave density with respect to the law of \(V\).

We apply this result in the setting of Theorem~\ref{thm:main2}. A finite-support approximation and rescaling reduce the proof to the case of a finitely supported random vector \(Z\) satisfying
\[
Z\preceq_{\mathrm{cx}}(1-\varepsilon)G_\Sigma
\qquad\text{for some }\varepsilon\in(0,1).
\]
Applying \cite[Proposition~3.3]{hua2026talagrand}, we obtain a martingale coupling of \(Z\) and \(G_\Sigma\) 
such that, for every \(z\in\supp(Z)\), the conditional law of \(G_\Sigma\) given \(Z=z\) has a log-concave density with respect to \(\mcN(0,\Sigma)\). Set \(R:=G_\Sigma-Z\). For
each integer \(r\ge1\), we define its conditional moment tensor by
\[
M_r(Z):=\E[R^{\otimes r}\mid Z].
\]

Since symmetrizing \(T\) does not change \(T(Z,\ldots,Z)\) and cannot
increase \(\|T\|_{\mathsf F}\), we may assume without loss of generality
that \(T\) is symmetric. For \(0\le r\le p\) and
\(B\in(\mathbb R^d)^{\otimes r}\), we write 
\[
T[Z^{p-r},B]
:=
\sum_{i_1,\ldots,i_r=1}^d
B_{i_1,\ldots,i_r}\,
T\bigl(
\underbrace{Z,\ldots,Z}_{p-r\text{ times}},
e_{i_1},\ldots,e_{i_r}
\bigr),
\]
where $(e_i)_{i=1}^d$ denote the standard basis vectors of $\R^d$. Then, using multilinearity and symmetry of \(T\),
and the fact that \(M_1(Z)=\E[R\mid Z]=0\), we obtain the expansion
\begin{align}\label{eq:proof2-expansion}
T[Z^p]=T(Z,\ldots,Z)
=
\underbrace{
\E\bigl[T(G_\Sigma,\ldots,G_\Sigma)\mid Z\bigr]
}_{\text{Gaussian contribution}}
-
\underbrace{
\sum_{r=2}^p \binom{p}{r}
T[Z^{p-r},M_r(Z)]
}_{\text{lower-degree contributions}}.
\end{align}
Thus the degree-\(p\) polynomial \(T(Z,\ldots,Z)\) is expressed as a
conditional expectation of the corresponding Gaussian polynomial, minus
correction terms of degree at most \(p-2\) in \(Z\), whose coefficients
\(M_r(Z)\) are the conditional moment tensors of the remainder \(R\). This key identity is rigorously established in Equation \eqref{eq:Lq-conditional-expansion}.

The following proposition bounds polynomial moments of Gaussian vectors
and is used to control the Gaussian contribution in
\eqref{eq:proof2-expansion}. We prove it in
Subsection~\ref{subsec:Hilbert-Gaussian-tensor-Lq} using Gaussian chaos
decomposition and moment bounds for the individual chaos components. 

\begin{proposition}[Polynomial moments for Gaussian vectors]
\label{prop:Hilbert-Gaussian-tensor-Lq}
Let \(G_\Sigma\sim\mcN(0,\Sigma)\), where
\(\Sigma\succeq0\). Then, for every integer \(p\ge1\), every \(q\ge1\),
every finite-dimensional real Hilbert space \(\mathcal H\), and every
\(p\)-linear map
\[
T:(\mathbb R^d)^p\to\mathcal H,
\]
one has
\[
\left(
\E
\big\|
T(G_\Sigma,\ldots,G_\Sigma)
\big\|_{\mathcal H}^q
\right)^{1/q}
\lesssim_p
\Psi_p(q,\Sigma)\|T\|_{\mathsf F},
\]
and
\[
\left(
\E
\big\|
T(G_\Sigma,\ldots,G_\Sigma)
-
\E T(G_\Sigma,\ldots,G_\Sigma)
\big\|_{\mathcal H}^q
\right)^{1/q}
\lesssim_p
\Psi_p^\circ(q,\Sigma)\|T\|_{\mathsf F},
\]
where $\Psi_{p}(q,\Sigma)$ and $\Psi_p^\circ(q,\Sigma)$ are defined in \eqref{eq:Psi_p}.
\end{proposition}

The next proposition bounds the Frobenius norm of the moment tensors of a
uniformly log-concave distribution; applied conditionally, it controls
\(M_r(Z)\) in \eqref{eq:proof2-expansion}. We prove it in Subsection~\ref{subsec:ULC} using Fathi's result on bounded Stein kernels for uniformly log-concave random vectors \cite{fathi2019stein}. 

\begin{proposition}[Moment tensors of uniformly log-concave distributions]
\label{prop:ULC-moment-tensors}
Let \(\Sigma\) be positive definite, and let \(Y\in\mathbb R^d\) be a
centered random vector with density \(c\exp(-V(y))\), where \(c>0\) is
the normalizing constant and \(V\in C^2(\mathbb R^d)\) satisfies
\(\nabla^2V(y)\succeq\Sigma^{-1}\) for every \(y\in\mathbb R^d\).
Then, for every integer \(p\ge2\),
\[
\left\|
\E Y^{\otimes p}
\right\|_{\mathsf F}
\lesssim_p
\rho_p(\Sigma),  \qquad
\rho_p(\Sigma)
:=
\begin{cases}
\|\Sigma\|_{\mathsf F}^{p/2},
& p \text{ even},\\[3pt]
\|\Sigma\|_{\mathsf F}^{(p-1)/2}\|\Sigma\|^{1/2},
& p \text{ odd}.
\end{cases}
\]
\end{proposition}

\begin{remark}
\label{rem:ULC-third-moment}
In a very recent work, Chen and Klartag determine the optimal universal constant in the
thin-shell inequality \cite{chen2026thinshell}.  Their Theorem 1.2  shows that every isotropic
log-concave random vector \(Y\in\mathbb R^d\) satisfies the sharp estimate
\[
\bigl\|\E Y^{\otimes 3}\bigr\|_{\mathsf F}^{2}
\le 4d.
\]
Equality is attained when the coordinates of $Y$ are independent standard, centered, exponential random variables. Thus, at order \(p=3\), a sharp estimate is available under mere
log-concavity in the isotropic setting. By contrast, under the stronger
assumption of uniform log-concavity,
Proposition~\ref{prop:ULC-moment-tensors} provides anisotropic bounds at
all orders needed for the structured martingale coupling argument in the
proof of Theorem~\ref{thm:main2}. It would be interesting to investigate
sharp dimension-free, covariance-dependent Frobenius-norm bounds for
higher-order moment tensors under mere log-concavity.
\end{remark}

We prove Theorem~\ref{thm:main2} by induction on the order \(p\). For
\(p=1\) the map \(z\mapsto\|T(z)\|_{\mathcal H}^q\) is convex, so the
estimate follows directly from convex domination and
Proposition~\ref{prop:Hilbert-Gaussian-tensor-Lq}. Let \(p\ge2\), and
suppose the estimate holds at every degree \(\le p-1\).

We start from the decomposition \eqref{eq:proof2-expansion}. By conditional
Jensen's inequality and Proposition~\ref{prop:Hilbert-Gaussian-tensor-Lq},
the Gaussian contribution satisfies
\begin{align}\label{eq:proof2-Gaussian}
\left(\E \bigl\|
\E\bigl[T(G_\Sigma,\ldots,G_\Sigma)\mid Z\bigr]
\bigr\|_{\mathcal H}^q\right)^{1/q}
\le
\left(\E \bigl\|T(G_\Sigma,\ldots,G_\Sigma)\bigr\|_{\mathcal H}^q\right)^{1/q}
\lesssim_p
\Psi_p(q,\Sigma)\|T\|_{\mathsf F}.
\end{align}

For the lower-degree contributions
\(T[Z^{p-r},M_r(Z)]\), we define the \((p-r)\)-linear map
\[
(\mathcal S_rT)(x_1,\ldots,x_{p-r})
:=
\sum_{i_1,\ldots,i_r=1}^d
T(x_1,\ldots,x_{p-r},e_{i_1},\ldots,e_{i_r})
\otimes e_{i_1}\otimes\cdots\otimes e_{i_r},
\]
with values in \(\mathcal H\otimes(\mathbb R^d)^{\otimes r}\). Then $T[Z^{p-r},M_r(Z)]
=
\bigl\langle
(\mathcal S_rT)(Z,\ldots,Z),\,M_r(Z)
\bigr\rangle_{\mathsf F}$,
where the pairing is taken over the last \(r\) tensor factors and leaves
the \(\mathcal H\)-component untouched. Moreover
\(\|\mathcal S_rT\|_{\mathsf F}=\|T\|_{\mathsf F}\), since \(\mathcal S_r\)
only rearranges the coefficients of \(T\). By Cauchy--Schwarz inequality,
\begin{align}\label{eq:proof2-pointwise}
\bigl\|T[Z^{p-r},M_r(Z)]\bigr\|_{\mathcal H}
\le
\|M_r(Z)\|_{\mathsf F}\,
\bigl\|(\mathcal S_rT)(Z,\ldots,Z)\bigr\|_{\mathcal H\otimes(\mathbb R^d)^{\otimes r}}.
\end{align}

Conditionally on \(Z\), the law of \(R\) is a translate of the conditional
law of \(G_\Sigma\) given \(Z\). Since the latter has a log-concave density
with respect to \(\mathcal N(0,\Sigma)\), the conditional law of \(R\) is
\(\Sigma\)-uniformly log-concave. Therefore,
Proposition~\ref{prop:ULC-moment-tensors} gives
\(\|M_r(Z)\|_{\mathsf F}\lesssim_p\rho_r(\Sigma)\) almost surely.

Taking \(L_q\)-norms on both sides of \eqref{eq:proof2-pointwise} and applying the induction hypothesis at
degree \(p-r\) to \(\mathcal S_rT\) yields
\begin{align}\label{eq:proof2-low-degree}
\left(\E\bigl\|T[Z^{p-r},M_r(Z)]\bigr\|_{\mathcal H}^q\right)^{1/q}
&\lesssim_p \rho_r(\Sigma) \left(\E \bigl\|(\mathcal S_rT)(Z,\ldots,Z)\bigr\|_{\mathcal H\otimes(\mathbb R^d)^{\otimes r}}^q\right)^{1/q} \nonumber\\
& \lesssim_p
\rho_r(\Sigma)\Psi_{p-r}(q,\Sigma)\|\mathcal{S}_r T\|_{\mathsf F}
\nonumber\\
&=
\rho_r(\Sigma)\Psi_{p-r}(q,\Sigma)\|T\|_{\mathsf F}
\nonumber\\
&\le
\Psi_p(q,\Sigma)\|T\|_{\mathsf F},
\end{align}
where the last step follows from the relation between
\(\rho_r(\Sigma)\) and \(\Psi_{p-r}(q,\Sigma)\) in 
\eqref{eq:rho-Psi-product}. 

Combining \eqref{eq:proof2-expansion}, \eqref{eq:proof2-Gaussian}, and
\eqref{eq:proof2-low-degree} gives
\[
\left(
\E
\big\|
T(Z,\ldots,Z)
\big\|_{\mathcal H}^q
\right)^{1/q}
\lesssim_p
\Psi_p(q,\Sigma)\|T\|_{\mathsf F},
\]
which proves the estimate at degree \(p\) and closes the induction. Note that the Hilbert-valued formulation of
Theorem~\ref{thm:main2} is used in the induction: even for scalar-valued \(T\), the
induction hypothesis is applied to the maps \(\mathcal S_rT\), which take
values in \(\mathcal H\otimes(\mathbb R^d)^{\otimes r}\).

\bigskip

Figure~\ref{fig:roadmap} provides a roadmap of the proofs of our main results.

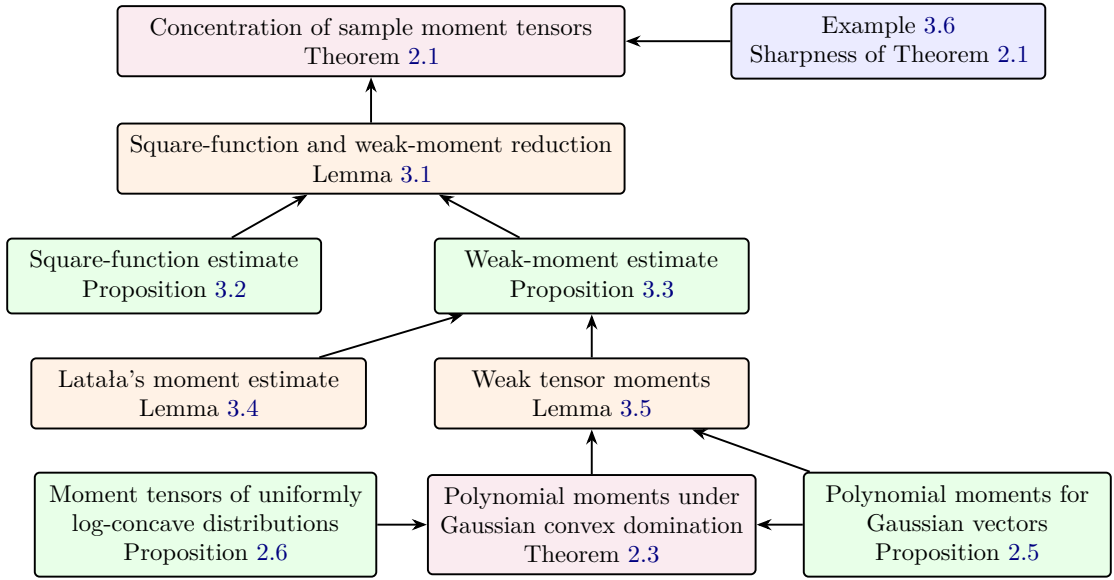
\begin{figure}[H]
\centering
\resizebox{0.9\textwidth}{!}{
\begin{tikzpicture}[
font=\small,
box/.style={
draw,
rounded corners=2pt,
align=center,
inner sep=5pt,
minimum height=9mm,
text width=42mm
},
smallbox/.style={
draw,
rounded corners=2pt,
align=center,
inner sep=4pt,
minimum height=8mm,
text width=37mm
},
result/.style={box, fill=purple!8, thick},
example/.style={box, fill=blue!8, thick},
reduction/.style={box, fill=orange!10, thick},
estimate/.style={box, fill=green!9, thick},
input/.style={box, fill=orange!10, thick},
arrow/.style={-{Stealth[length=2.2mm]}, thick},
group/.style={
draw,
dashed,
rounded corners=3pt,
inner sep=7pt
}
]

\node[result, text width=70mm] (upper) at (0,0)
{Concentration of sample moment tensors \\Theorem~\ref{thm:main1}};

\node[reduction, text width=70mm] (hilbert) at (0,-1.7)
{Square-function and weak-moment reduction\\
Lemma~\ref{lem:Hilbert-moment-reduction}};

\node[estimate] (square) at (-3,-3.4)
{Square-function estimate\\
Proposition~\ref{prop:master-common-diagonal}};

\node[estimate] (scalar) at (3.2,-3.4)
{Weak-moment estimate\\
Proposition~\ref{prop:master-scalar-projections}};

\node[input,text width=45mm] (latala) at (-2.5,-5.1)
{Lata{\l}a's moment estimate \\
Lemma~\ref{lem:latala-iid-centered-sum}};

\node[input] (subdir) at (3.2,-5.1)
{Weak tensor moments\\
Lemma~\ref{lem:subgaussian-directional-tensor-moments}};

\node[estimate, text width=46mm] (log-concave) at (-2.4,-7)
{Moment tensors of uniformly\\
log-concave distributions\\
Proposition~\ref{prop:ULC-moment-tensors}};

\node[result, text width=44mm] (convex) at (3.2,-7)
{Polynomial moments under\\
Gaussian convex domination\\
Theorem~\ref{thm:main2}};

\node[estimate, text width=41mm] (wick) at (8.5,-7)
{Polynomial moments for \\ Gaussian vectors \\
Proposition~\ref{prop:Hilbert-Gaussian-tensor-Lq}};

\node[example] (subsharp) at (7.5, 0)
{Example~\ref{example:subgaussian_tensor_sharpness}\\Sharpness of Theorem \ref{thm:main1}
};

\draw[arrow] (hilbert) -- (upper);
\draw[arrow] (square) -- (hilbert);
\draw[arrow] (scalar) -- (hilbert);
\draw[arrow] (latala) -- (scalar);
\draw[arrow] (subdir) -- (scalar);
\draw[arrow] (convex) -- (subdir);
\draw[arrow] (wick) -- (convex);
\draw[arrow] (log-concave) -- (convex);
\draw[arrow] (subsharp) -- (upper);
\draw[arrow] (wick) -- (subdir);
\end{tikzpicture}
}
\caption{Roadmap of the proofs of Theorems~\ref{thm:main1} and~\ref{thm:main2}.}
\label{fig:roadmap}
\end{figure}

\section{Concentration of sample moment tensors}\label{sec:proof1}

\subsection{Proof of Theorem~\ref{thm:main1}}

Our proof begins with the following reduction, which relates strong
moments of Hilbert-valued sums to a square-function moment and weak moments of
scalar projections. Its proof combines the Hilbert-space decoupling inequality
\cite[Exercise~6.1.4]{vershynin2018high} with a symmetrization
argument. We include the details in
Subsection~\ref{subsec:Hilbert-moment-reduction} for completeness.

\begin{lemma}[Square-function and weak-moment reduction for Hilbert-valued sums]
\label{lem:Hilbert-moment-reduction}
Let \(Z_1,\ldots,Z_N\) be independent centered random vectors in a real
Hilbert space \(\mathcal H\), let \(q\ge2\), and assume that
\[
\|Z_i\|_{L_q(\mathcal H)}
:=
\left(
\E\|Z_i\|_{\mathcal H}^q
\right)^{1/q}
<\infty,
\qquad i=1,\ldots,N.
\]
Then
\[
\bigg\|
\sum_{i=1}^N Z_i
\bigg\|_{L_q(\mathcal H)}
\asymp
\left\|
\bigg(
\sum_{i=1}^N \|Z_i\|_{\mathcal H}^2
\bigg)^{1/2}
\right\|_{L_q}
+
\sup_{\|h\|_{\mathcal H}=1}
\left\|
\sum_{i=1}^N
\langle h,Z_i\rangle_{\mathcal H}
\right\|_{L_{q/2}}.
\]
\end{lemma}

It is important that the implicit constant in Lemma~\ref{lem:Hilbert-moment-reduction} is universal and, in particular,
independent of \(q\). We apply Lemma~\ref{lem:Hilbert-moment-reduction} in the
Hilbert tensor space \((\mathbb R^d)^{\otimes p}\), equipped with the Frobenius
inner product, to the independent centered tensors
\[
Z_i:=X_i^{\otimes p}-\E X^{\otimes p},
\qquad i=1,\ldots,N.
\]
The following proposition controls the square-function term in both the
sub-Gaussian and Gaussian settings. Its proof is given in
Subsection~\ref{subsec:Square-function-estimate}.

\begin{proposition}[Square-function estimate]
\label{prop:master-common-diagonal}
Let \(X_1,\ldots,X_N\) be i.i.d. copies of a centered sub-Gaussian random
vector \(X\in\mathbb R^d\) with covariance matrix \(\Sigma\) and
sub-Gaussian constant \(K\). Let \(p\ge2\) be an integer. Then, for every
\(q\ge1\),
\begin{equation}\label{eq:master-diagonal-subg}
\left\|
\bigg(
\sum_{i=1}^N
\big\|
X_i^{\otimes p}-\E X^{\otimes p}
\big\|_{\mathsf F}^2
\bigg)^{1/2}
\right\|_{L_q}
\lesssim_p
K^p
\left(
\sqrt N\,\tr(\Sigma)^{p/2}
+
\|\Sigma\|^{p/2}q^{p/2}
\right).
\end{equation}
Moreover, if \(X\) is Gaussian, then
\begin{equation}\label{eq:master-diagonal-g}
\left\|
\bigg(
\sum_{i=1}^N
\big\|
X_i^{\otimes p}-\E X^{\otimes p}
\big\|_{\mathsf F}^2
\bigg)^{1/2}
\right\|_{L_q}
\asymp_p
\sqrt N\,\tr(\Sigma)^{p/2}
+
\|\Sigma\|^{p/2}q^{p/2}.
\end{equation}
\end{proposition}

The weak-moment term in Lemma~\ref{lem:Hilbert-moment-reduction} involves
scalar projections of the tensor-valued sum. The following proposition
provides the required estimates in the sub-Gaussian and Gaussian cases,
giving rise to the intermediate scales \(\rho_p(\Sigma)\) and
\(\rho_{p,G}(\Sigma)\), respectively. Its proof is given in
Subsection~\ref{subsec:weak-moment-estimate}.

\begin{proposition}[Weak-moment estimate]
\label{prop:master-scalar-projections}
Let \(X_1,\ldots,X_N\) be i.i.d. copies of a centered sub-Gaussian random
vector \(X\in\mathbb R^d\) with covariance matrix \(\Sigma\) and
sub-Gaussian constant \(K\). Let \(p\ge2\) be an integer. Then, for every
\(q\ge1\),
\begin{align}
\sup_{\|A\|_{\mathsf F}\le1}\bigg\|
\sum_{i=1}^N
\left\langle
A,X_i^{\otimes p}-\E X^{\otimes p}
\right\rangle_{\mathsf F}
\bigg\|_{L_q}
\lesssim_p K^p
\begin{cases}
\rho_p(\Sigma)\sqrt{N(q\land N)}
+
\|\Sigma\|^{p/2}q^{p/2},
& p \text{ even},\\[3pt]
\rho_p(\Sigma)\sqrt{Nq}
+
\|\Sigma\|^{p/2}q^{p/2},
& p \text{ odd},
\end{cases}
\label{eq:master-scalar-subg}
\end{align}
where \(\rho_p(\Sigma)\) is as in Theorem~\ref{thm:main1}. Moreover, if \(X\) is Gaussian, then
\begin{align}
\sup_{\|A\|_{\mathsf F}\le1}\bigg\|
\sum_{i=1}^N
\left\langle
A,X_i^{\otimes p}-\E X^{\otimes p}
\right\rangle_{\mathsf F}
\bigg\|_{L_q}
\asymp_p
\begin{cases}
\rho_{p,G}(\Sigma)\left(\sqrt{Nq}+q\right)
+
\|\Sigma\|^{p/2}q^{p/2},
& p \text{ even},\\[3pt]
\rho_{p,G}(\Sigma)\sqrt{Nq}
+
\|\Sigma\|^{p/2}q^{p/2},
& p \text{ odd},
\end{cases}
\label{eq:master-scalar-Gaussian}
\end{align}
where \(\rho_{p,G}(\Sigma)\) is as in Theorem~\ref{thm:main1}.
\end{proposition}

We now combine Lemma~\ref{lem:Hilbert-moment-reduction} and
Propositions~\ref{prop:master-common-diagonal}
and~\ref{prop:master-scalar-projections} to prove Theorem \ref{thm:main1}.

\begin{proof}[Proof of Theorem~\ref{thm:main1}]
We first consider \(q\ge2\). Applying
Lemma~\ref{lem:Hilbert-moment-reduction} in
\((\mathbb R^d)^{\otimes p}\), equipped with the Frobenius inner product,
to the independent centered tensors
\[
X_i^{\otimes p}-\E X^{\otimes p},
\qquad i=1,\ldots,N,
\]
and then dividing the resulting equivalence by \(N\), we obtain
\begin{align}
\left(
\E
\bigg\|
\frac1N\sum_{i=1}^N
\left(
X_i^{\otimes p}-\E X^{\otimes p}
\right)
\bigg\|_{\mathsf F}^q
\right)^{1/q}
& \asymp
\frac1N
\left\|
\bigg(
\sum_{i=1}^N
\big\|
X_i^{\otimes p}-\E X^{\otimes p}
\big\|_{\mathsf F}^2
\bigg)^{1/2}
\right\|_{L_q}\nonumber\\
&\quad +
\frac1N
\sup_{\|A\|_{\mathsf F}\le1}
\bigg\|
\sum_{i=1}^N
\left\langle
A,
X_i^{\otimes p}-\E X^{\otimes p}
\right\rangle_{\mathsf F}
\bigg\|_{L_{q/2}}.
\label{eq:master-apply-decoupling}
\end{align}
In applying Proposition~\ref{prop:master-scalar-projections} below,
replacing \(q/2\) by \(q\) on its right-hand side changes the estimates
only by constants depending on \(p\).

\medskip
\noindent\textbf{Sub-Gaussian case.} 
\medskip

Combining \eqref{eq:master-apply-decoupling} with the sub-Gaussian parts
of Propositions~\ref{prop:master-common-diagonal}
and~\ref{prop:master-scalar-projections} proves the asserted bound
directly when \(p\) is even. When \(p\ge3\) is odd, the same estimates
give the asserted bound except that $\rho_p(\Sigma)\sqrt{\frac{q\land N}{N}}$
is replaced by $\rho_p(\Sigma)\sqrt{\frac qN}$.
If \(q\le N\), these quantities coincide. If \(q>N\), then
\[
\rho_p(\Sigma)^2
\le
\tr(\Sigma)^{p/2}\|\Sigma\|^{p/2},
\qquad
Nq\le\sqrt N\,q^{p/2},
\]
where we used
\(\|\Sigma\|_{\mathsf F}^2\le\tr(\Sigma)\|\Sigma\|\),
\(\|\Sigma\|\le\tr(\Sigma)\), and \(p\ge3\). Consequently,
\begin{align*}
\rho_p(\Sigma)\sqrt{\frac qN}
\le
\left(
\frac{\tr(\Sigma)^{p/2}}{\sqrt N}
\right)^{1/2}
\left(
\frac{\|\Sigma\|^{p/2}q^{p/2}}{N}
\right)^{1/2}\le
\frac12
\left(
\frac{\tr(\Sigma)^{p/2}}{\sqrt N}
+
\frac{\|\Sigma\|^{p/2}q^{p/2}}{N}
\right).
\end{align*}
Thus the untruncated term is absorbed by the trace and large-deviation
terms when \(q>N\). We therefore obtain
\begin{align*}
\left(
\E
\bigg\|
\frac1N\sum_{i=1}^N X_i^{\otimes p}
-\E X^{\otimes p}
\bigg\|_{\mathsf F}^q
\right)^{1/q}
\lesssim_p
K^p
\left(
\frac{\tr(\Sigma)^{p/2}}{\sqrt N}
+
\rho_p(\Sigma)\sqrt{\frac{q\land N}{N}}
+
\frac{\|\Sigma\|^{p/2}q^{p/2}}{N}
\right).
\end{align*}

\medskip
\noindent\textbf{Gaussian case.}
\medskip

Combining \eqref{eq:master-apply-decoupling} with the Gaussian parts of
Propositions~\ref{prop:master-common-diagonal}
and~\ref{prop:master-scalar-projections} proves the asserted two-sided
estimate directly when \(p\) is odd. When \(p\) is even, the same
argument gives the asserted estimate together with the additional term $\rho_{p,G}(\Sigma)q/N$.
Since $\rho_{p,G}(\Sigma)
=
\|\Sigma\|_{\mathsf F}^{p/2-1}\|\Sigma\|$,
we have
\begin{align*}
\frac{\rho_{p,G}(\Sigma)q}{N}
\le
\frac{\tr(\Sigma)^{p/2-1}\|\Sigma\|q}{N}\lesssim_p
\frac{\tr(\Sigma)^{p/2}}{N}
+
\frac{\|\Sigma\|^{p/2}q^{p/2}}{N}\le
\frac{\tr(\Sigma)^{p/2}}{\sqrt N}
+
\frac{\|\Sigma\|^{p/2}q^{p/2}}{N}.
\end{align*}
Hence the additional term is absorbed, and
\begin{align*}
\left(
\E
\bigg\|
\frac1N\sum_{i=1}^N X_i^{\otimes p}
-\E X^{\otimes p}
\bigg\|_{\mathsf F}^q
\right)^{1/q}
\asymp_p
\frac{\tr(\Sigma)^{p/2}}{\sqrt N}
+
\rho_{p,G}(\Sigma)\sqrt{\frac qN}
+
\frac{\|\Sigma\|^{p/2}q^{p/2}}{N}.
\end{align*}
This proves the claims for \(q\ge2\).

Finally, let \(1\le q<2\). By monotonicity and the comparability of the
corresponding rates at \(q\) and \(2\), the upper bounds follow from the case
\(q=2\). In the Gaussian case, finite-chaos moment comparison
\cite[Theorem~5.10]{Janson1997GaussianHilbertSpaces} gives
\[
\E\bigg\|
\sum_{i=1}^N
\left(X_i^{\otimes p}-\E X^{\otimes p}\right)
\bigg\|_{\mathsf F}
\gtrsim_p
\left(
\E\bigg\|
\sum_{i=1}^N
\left(X_i^{\otimes p}-\E X^{\otimes p}\right)
\bigg\|_{\mathsf F}^2
\right)^{1/2}.
\]
Monotonicity and rate comparability then extend the Gaussian lower bound from
\(q=2\) to \(1\le q<2\). This completes the proof.
\end{proof}

\subsection{Square-function estimate}\label{subsec:Square-function-estimate}

\begin{proof}[Proof of Proposition~\ref{prop:master-common-diagonal}]

\medskip
\noindent\textbf{Sub-Gaussian upper bound.}
\medskip

For \(i=1,\ldots,N\), set
\[
W_i
:=
\big\|
X_i^{\otimes p}-\E X^{\otimes p}
\big\|_{\mathsf F}^2.
\]
Then, for \(q\ge2\),
\[
\left\|
\bigg(
\sum_{i=1}^N
\big\|
X_i^{\otimes p}-\E X^{\otimes p}
\big\|_{\mathsf F}^2
\bigg)^{1/2}
\right\|_{L_q}
=
\bigg\|
\sum_{i=1}^N W_i
\bigg\|_{L_{q/2}}^{1/2}.
\]

The positive semi-definite quadratic-form inequality
\cite[Theorem~2.1 and Remark~2.2]{hsu2012tail}, see also
\cite[Lemma~2.7 and Remark~2.8]{zhivotovskiy2024dimension}, implies that,
for every \(t>0\), with probability at least \(1-e^{-t}\),
\[
\|X\|_2
\lesssim
K
\sqrt{
\tr(\Sigma)+t\|\Sigma\|
}.
\]
Integrating this tail bound gives
\begin{equation*}\label{eq:subg-Euclidean-moment}
\|\|X\|_2\|_{L_s}
\lesssim
K
\sqrt{
\tr(\Sigma)+s\|\Sigma\|
},
\qquad s\ge1.
\end{equation*}

Since
\(\|x^{\otimes p}\|_{\mathsf F}=\|x\|_2^p\), the triangle inequality and Jensen's inequality give
\[
\big\|
X^{\otimes p}-\E X^{\otimes p}
\big\|_{\mathsf F}
\le
\|X\|_2^p+\E\|X\|_2^p.
\]
Consequently, for every \(s\ge1\),
\begin{align}
\|W_1\|_{L_s}
=
\left\|
\big\|
X_1^{\otimes p}-\E X^{\otimes p}
\big\|_{\mathsf F}
\right\|_{L_{2s}}^2
\le
4\|\|X\|_2\|_{L_{2ps}}^{2p}
\lesssim_p
K^{2p}
\left(
\tr(\Sigma)^p+s^p\|\Sigma\|^p
\right).
\label{eq:Wi-moment-short}
\end{align}

Fix \(q\ge2\) and set \(r:=q/2\). Applying Lata{\l}a's moment estimate
for sums of i.i.d.\ nonnegative random variables
\cite[Corollary~1]{latala1997estimation} with moment parameter \(r\), we
obtain
\begin{equation}\label{eq:Latala-square-function}
\bigg\|
\sum_{i=1}^N W_i
\bigg\|_{L_r}
\asymp
\sup_{\max\{1,r/N\}\le s\le r}
\frac{r}{s}
\left(
\frac Nr
\right)^{1/s}
\|W_1\|_{L_s}.
\end{equation}
Substituting \eqref{eq:Wi-moment-short} into
\eqref{eq:Latala-square-function}, optimizing over \(s\), and using
\(\|\Sigma\|\le\tr(\Sigma)\), we obtain
\[
\bigg\|
\sum_{i=1}^N W_i
\bigg\|_{L_r}
\lesssim_p
K^{2p}
\Big(
N\tr(\Sigma)^p
+
r^p\|\Sigma\|^p
\Big).
\]
Since \(r=q/2\), it follows that
\begin{align*}
\left\|
\bigg(
\sum_{i=1}^N
\big\|
X_i^{\otimes p}-\E X^{\otimes p}
\big\|_{\mathsf F}^2
\bigg)^{1/2}
\right\|_{L_q}=
\bigg\|
\sum_{i=1}^N W_i
\bigg\|_{L_{q/2}}^{1/2}
\lesssim_p
K^p
\left(
\sqrt N\,\tr(\Sigma)^{p/2}
+
\|\Sigma\|^{p/2}q^{p/2}
\right).
\end{align*}

For \(1\le q<2\), the same conclusion follows from the case \(q=2\)
and monotonicity of \(L_q\)-norms. This completes the proof of the upper bound for sub-Gaussian random vectors.

\medskip
\noindent \textbf{Gaussian upper bound.}
\medskip

When \(X\) is Gaussian, the upper bound follows immediately from the sub-Gaussian estimate above, since a centered Gaussian random vector satisfies \eqref{eq:def-subg} with \(K=\sqrt{8/3}\).

\medskip
\noindent\textbf{Gaussian lower bound.}
\medskip

Suppose now that \(X=G\sim\mcN(0,\Sigma)\). We first treat \(q\ge2\) and set
\(r:=q/2\ge1\).

\medskip
\textbf{Step 1: a lower bound for \(\|W_1\|_{L_s}\).} We claim that
\begin{equation}\label{eq:gaussian-W1-lower}
\|W_1\|_{L_s}
\gtrsim_p
\tr(\Sigma)^p+s^p\|\Sigma\|^p,
\qquad s\ge1.
\end{equation}
Since
\[
\|W_1\|_{L_s}
=
\left\|
\big\|G^{\otimes p}-\E G^{\otimes p}\big\|_{\mathsf F}
\right\|_{L_{2s}}^2,
\]
it suffices to bound the latter norm below by
\(\tr(\Sigma)^{p/2}\) and by \(s^{p/2}\|\Sigma\|^{p/2}\) separately.

Let \(v_1\) be a unit top eigenvector of \(\Sigma\) and set
\(\xi_1:=\|\Sigma\|^{-1/2}\langle G,v_1\rangle\sim\mcN(0,1)\). Testing
against \(A:=v_1^{\otimes p}\), which satisfies \(\|A\|_{\mathsf F}=1\),
gives
\[
\big\|G^{\otimes p}-\E G^{\otimes p}\big\|_{\mathsf F}
\ge
\left|
\left\langle A,G^{\otimes p}-\E G^{\otimes p}\right\rangle_{\mathsf F}
\right|
=
\|\Sigma\|^{p/2}
\left|\xi_1^p-\E\xi_1^p\right|.
\]
Since \(\|\xi_1^p-\E\xi_1^p\|_{L_{2s}}\asymp_p s^{p/2}\) for \(s\ge1\),
\begin{equation}\label{eq:gaussian-W1-tail-scale}
\left\|
\big\|G^{\otimes p}-\E G^{\otimes p}\big\|_{\mathsf F}
\right\|_{L_{2s}}
\gtrsim_p
\|\Sigma\|^{p/2}s^{p/2}.
\end{equation}

For the trace scale, recall that by Wick's formula \(\E G^{\otimes p}=0\)
when \(p\) is odd, while for \(p\) even \(\E G^{\otimes p}\) is a sum of
\((p-1)!!\) tensors, each obtained by permuting the legs of
\(\Sigma^{\otimes p/2}\) and therefore of Frobenius norm
\(\|\Sigma\|_{\mathsf F}^{p/2}\). Hence, with \(b_p:=(p-1)!!\),
\begin{equation}\label{eq:gaussian-mean-tensor-upper}
\big\|\E G^{\otimes p}\big\|_{\mathsf F}
\le
b_p\|\Sigma\|_{\mathsf F}^{p/2}.
\end{equation}
Since \(\|x^{\otimes p}\|_{\mathsf F}=\|x\|_2^p\), orthogonality of the
mean and the fluctuation in \((\mathbb R^d)^{\otimes p}\) together with
Jensen's inequality gives
\begin{align}
\E\big\|G^{\otimes p}-\E G^{\otimes p}\big\|_{\mathsf F}^2
&=
\E\|G\|_2^{2p}
-
\big\|\E G^{\otimes p}\big\|_{\mathsf F}^2\nonumber\\
&\ge
\left(\E\|G\|_2^2\right)^p
-
b_p^2\|\Sigma\|_{\mathsf F}^p
=
\tr(\Sigma)^p-b_p^2\|\Sigma\|_{\mathsf F}^p.
\label{eq:gaussian-W1-L2-identity}
\end{align}
By \(\|\Sigma\|_{\mathsf F}^2\le\|\Sigma\|\tr(\Sigma)\),
\[
b_p^2\|\Sigma\|_{\mathsf F}^p
\le
b_p^2
\tr(\Sigma)^p
\left(
\frac{\|\Sigma\|}{\tr(\Sigma)}
\right)^{p/2}.
\]
Set \(\kappa_p:=(2b_p^2)^{2/p}\). If \(\tr(\Sigma)\ge\kappa_p\|\Sigma\|\),
the right-hand side is at most \(\tfrac12\tr(\Sigma)^p\), so
\eqref{eq:gaussian-W1-L2-identity} and \(2s\ge2\) yield
\begin{equation}\label{eq:gaussian-W1-trace-scale}
\left\|
\big\|G^{\otimes p}-\E G^{\otimes p}\big\|_{\mathsf F}
\right\|_{L_{2s}}
\ge
\left\|
\big\|G^{\otimes p}-\E G^{\otimes p}\big\|_{\mathsf F}
\right\|_{L_2}
\ge
2^{-1/2}\tr(\Sigma)^{p/2}.
\end{equation}
If instead \(\tr(\Sigma)<\kappa_p\|\Sigma\|\), then, since \(s\ge1\),
\[
\tr(\Sigma)^{p/2}
\le
\kappa_p^{p/2}\|\Sigma\|^{p/2}
\le
\kappa_p^{p/2}s^{p/2}\|\Sigma\|^{p/2},
\]
so the trace scale is already controlled by
\eqref{eq:gaussian-W1-tail-scale}. In either case, combining
\eqref{eq:gaussian-W1-tail-scale} and
\eqref{eq:gaussian-W1-trace-scale} and squaring proves
\eqref{eq:gaussian-W1-lower}.

\medskip
\textbf{Step 2: the two-sided Latała's estimate.} We use
\eqref{eq:Latala-square-function}, which is an equivalence, and bound the
supremum from below by evaluating it at two admissible values of \(s\).

Take first \(s:=\max\{1,r/N\}\). If \(r\le N\), then \(s=1\) and the
corresponding term equals \(N\|W_1\|_{L_1}\). If \(r>N\), then
\(s=r/N\) and the corresponding term equals
\[
N\left(\frac Nr\right)^{N/r}\|W_1\|_{L_{r/N}}
\ge
e^{-1/e}N\|W_1\|_{L_1},
\]
where we used \(\inf_{0<x\le1}x^x=e^{-1/e}\) and monotonicity of
\(L_s\)-norms. In both cases, \eqref{eq:gaussian-W1-lower} with \(s=1\)
gives a lower bound \(\gtrsim_p N\tr(\Sigma)^p\).

Taking instead \(s:=r\), and using \(N\ge1\) together with
\(\max_{r\ge1}r^{-1}\log r=1/e\), we obtain
\[
\left(\frac Nr\right)^{1/r}\|W_1\|_{L_r}
\ge
e^{-1/e}\|W_1\|_{L_r}
\gtrsim_p
r^p\|\Sigma\|^p.
\]
Consequently,
\[
\bigg\|\sum_{i=1}^N W_i\bigg\|_{L_r}
\gtrsim_p
N\tr(\Sigma)^p+r^p\|\Sigma\|^p,
\]
and, taking square roots with \(r=q/2\),
\[
\left\|
\bigg(
\sum_{i=1}^N
\big\|G_i^{\otimes p}-\E G^{\otimes p}\big\|_{\mathsf F}^2
\bigg)^{1/2}
\right\|_{L_q}
\gtrsim_p
\sqrt N\,\tr(\Sigma)^{p/2}
+
\|\Sigma\|^{p/2}q^{p/2}.
\]

\medskip
\textbf{Step 3: the range \(1\le q<2\).} The upper bound follows from the
case \(q=2\) by monotonicity. For the lower bound, note that
\[
V:=
\bigg(
\sum_{i=1}^N
\big\|G_i^{\otimes p}-\E G^{\otimes p}\big\|_{\mathsf F}^2
\bigg)^{1/2}
=
\big\|
\left(G_i^{\otimes p}-\E G^{\otimes p}\right)_{i=1}^N
\big\|_{\ell_2^N((\mathbb R^d)^{\otimes p})},
\]
and that
\(\left(G_i^{\otimes p}-\E G^{\otimes p}\right)_{i=1}^N\) is a
Hilbert-space-valued Gaussian chaos of order at most \(p\) in the
Gaussian vector \((G_1,\ldots,G_N)\). Finite-chaos moment comparison
\cite[Theorem~5.10]{Janson1997GaussianHilbertSpaces} therefore gives
\[
\|V\|_{L_2}
\lesssim_p
\|V\|_{L_1}
\le
\|V\|_{L_q}.
\]
Since the right-hand side of \eqref{eq:master-diagonal-g} at \(q=2\) is
comparable, up to constants depending only on \(p\), to its value at
\(q\in[1,2)\), the lower bound extends to \(1\le q<2\). This completes
the proof.
\end{proof}

\subsection{Weak-moment estimate}\label{subsec:weak-moment-estimate}

The proof of Proposition \ref{prop:master-scalar-projections} combines Lata{\l}a's moment estimate for sums of independent
random variables with weak moment bounds for a single centered tensor
power.

\begin{lemma}[{\cite[Corollary~2 and Remark~2]{latala1997estimation}}]
\label{lem:latala-iid-centered-sum}
Let \(Y_1,\ldots,Y_N\) be independent copies of a centered real-valued
random variable \(Y\). Then, for every \(q\ge2\) such that \(Y\in L_q\),
\begin{equation}\label{eq:latala-iid-centered-sum}
\bigg\|
\sum_{i=1}^N Y_i
\bigg\|_{L_q}
\asymp
\sup_{\max\{2,q/N\}\le s\le q}
\frac qs
\left(\frac Nq\right)^{1/s}
\|Y\|_{L_s}.
\end{equation}
\end{lemma}

It remains to control the \(q\)-th moment of the centered tensor,
\[
\left\langle
A,X^{\otimes p}-\E X^{\otimes p}
\right\rangle_{\mathsf F},
\qquad
\|A\|_{\mathsf F}=1.
\]
The required estimates are given by the following lemma, whose proof is given in Subsection \ref{subsec:subgaussian-directional-tensor-moments}.

\begin{lemma}[Weak moments of a single tensor power]
\label{lem:subgaussian-directional-tensor-moments}
Let \(X\in\mathbb R^d\) be a centered sub-Gaussian random vector with
covariance matrix \(\Sigma\) and sub-Gaussian constant \(K\), and let
\(p\ge2\) be an integer. Then, for every \(q\ge1\),
\[
\sup_{\|A\|_{\mathsf F}=1}
\left\|
\left\langle
A,X^{\otimes p}-\E X^{\otimes p}
\right\rangle_{\mathsf F}
\right\|_{L_q}
\lesssim_p
K^p \begin{cases}
\|\Sigma\|_{\mathsf F}^{p/2}+\|\Sigma\|^{p/2}q^{p/2},
& p \text{ even},\\[3pt]
\|\Sigma\|_{\mathsf F}^{(p-1)/2}\|\Sigma\|^{1/2}\sqrt q+\|\Sigma\|^{p/2}q^{p/2},
& p \text{ odd}.
\end{cases}
\]
 Moreover, if \(X\) is Gaussian, then, for every \(q\ge1\),
\[
\sup_{\|A\|_{\mathsf F}=1}
\left\|
\left\langle
A,X^{\otimes p}-\E X^{\otimes p}
\right\rangle_{\mathsf F}
\right\|_{L_q}
\asymp_p
\begin{cases}
\|\Sigma\|_{\mathsf F}^{p/2-1}\|\Sigma\| q
+
\|\Sigma\|^{p/2}q^{p/2},
& p \text{ even},\\[3pt]
\|\Sigma\|_{\mathsf F}^{(p-1)/2}\|\Sigma\|^{1/2}\sqrt q
+
\|\Sigma\|^{p/2}q^{p/2},
& p \text{ odd},
\end{cases}
\]
and 
\[
\sup_{\|A\|_{\mathsf F}=1}
\left\|
\left\langle
A,X^{\otimes p}
\right\rangle_{\mathsf F}
\right\|_{L_q}
\asymp_p
\begin{cases}
\|\Sigma\|_{\mathsf F}^{p/2}
+
\|\Sigma\|^{p/2}q^{p/2},
& p \text{ even},\\[3pt]
\|\Sigma\|_{\mathsf F}^{(p-1)/2}\|\Sigma\|^{1/2}\sqrt q
+
\|\Sigma\|^{p/2}q^{p/2},
& p \text{ odd}.
\end{cases}
\]
\end{lemma}

\begin{proof}[Proof of Proposition~\ref{prop:master-scalar-projections}]
Throughout the proof, set
\[
Z_i:=X_i^{\otimes p}-\E X^{\otimes p},
\qquad i=1,\ldots,N,
\qquad
Z:=X^{\otimes p}-\E X^{\otimes p}.
\]
We first consider \(q\ge2\). For \(\alpha\ge0\), define
\[
\mathcal L_\alpha(N,q)
:=
\sup_{\max\{2,q/N\}\le s\le q}
\frac qs
\left(
\frac Nq
\right)^{1/s}
s^\alpha .
\]
An elementary optimization gives
\begin{equation}
\mathcal L_0(N,q)
\asymp
\sqrt{N(q\land N)},
\qquad
\mathcal L_{1/2}(N,q)
\asymp
\sqrt{Nq},
\label{eq:elementary-latala-optimization}
\end{equation}
and, for every \(\alpha\ge1\),
\begin{equation}
\mathcal L_\alpha(N,q)
\asymp_\alpha
\sqrt{Nq}+q^\alpha .
\label{eq:elementary-latala-optimization-power}
\end{equation}
We shall also use the fact that, for nonnegative \(f_s\) and
\(g_s\),
\begin{equation}
\sup_s(f_s+g_s)\asymp\sup_s f_s+\sup_s g_s.
\label{eq:sup-additivity}
\end{equation}

\medskip
\noindent\textbf{Sub-Gaussian upper bound.}
\medskip

Recall that
\[
\rho_p(\Sigma)
=
\begin{cases}
\|\Sigma\|_{\mathsf F}^{p/2},
& p \text{ even},\\[3pt]
\|\Sigma\|_{\mathsf F}^{(p-1)/2}\|\Sigma\|^{1/2},
& p \text{ odd}.
\end{cases}
\]
By Lemma~\ref{lem:latala-iid-centered-sum}, followed by interchanging
the suprema over \(A\) and \(s\), we have
\begin{align*}
\sup_{\|A\|_{\mathsf F}=1}
\left\|
\sum_{i=1}^N \langle A,Z_i\rangle_{\mathsf F}
\right\|_{L_q}
&\asymp
\sup_{\max\{2,q/N\}\le s\le q}
\frac{q}{s}
\left(\frac{N}{q}\right)^{1/s}
\sup_{\|A\|_{\mathsf F}=1}
\left\|
\langle A,Z\rangle_{\mathsf F}
\right\|_{L_s}.
\end{align*}
Applying the sub-Gaussian bound in
Lemma~\ref{lem:subgaussian-directional-tensor-moments} and
\eqref{eq:sup-additivity}, we therefore obtain
\begin{align*}
\sup_{\|A\|_{\mathsf F}=1}
\left\|
\sum_{i=1}^N \langle A,Z_i\rangle_{\mathsf F}
\right\|_{L_q}
\lesssim_p
K^p
\begin{cases}
\rho_p(\Sigma)\mathcal L_0(N,q)
+
\|\Sigma\|^{p/2}\mathcal L_{p/2}(N,q),
& p \text{ even},\\[3pt]
\rho_p(\Sigma)\mathcal L_{1/2}(N,q)
+
\|\Sigma\|^{p/2}\mathcal L_{p/2}(N,q),
& p \text{ odd}.
\end{cases}
\end{align*}
Since \(\|\Sigma\|^{p/2}\le\rho_p(\Sigma)\),
\eqref{eq:elementary-latala-optimization} and
\eqref{eq:elementary-latala-optimization-power} imply
\begin{align}
\sup_{\|A\|_{\mathsf F}=1}
\bigg\|
\sum_{i=1}^N\langle A,Z_i\rangle_{\mathsf F}
\bigg\|_{L_q}
\lesssim_p
K^p
\begin{cases}
\rho_p(\Sigma)\sqrt{N(q\land N)}
+
\|\Sigma\|^{p/2}q^{p/2},
& p \text{ even},\\[3pt]
\rho_p(\Sigma)\sqrt{Nq}
+
\|\Sigma\|^{p/2}q^{p/2},
& p \text{ odd}.
\end{cases}
\label{eq:subgaussian-scalar-sum}
\end{align}
Indeed, the additional term \(\|\Sigma\|^{p/2}\sqrt{Nq}\) arising from
\eqref{eq:elementary-latala-optimization-power} is absorbed as follows.
When \(p\) is even, it is absorbed by the first term if \(q\le N\), and by
\(\|\Sigma\|^{p/2}q^{p/2}\) if \(q>N\). When \(p\) is odd, it is absorbed
by \(\rho_p(\Sigma)\sqrt{Nq}\).

\medskip
\noindent\textbf{Gaussian two-sided bound.}
\medskip

Suppose now that \(X\) is Gaussian. Repeating the preceding argument
with the Gaussian two-sided estimate in
Lemma~\ref{lem:subgaussian-directional-tensor-moments}, we obtain
\begin{align*}
\sup_{\|A\|_{\mathsf F}=1}
\bigg\|
\sum_{i=1}^N\langle A,Z_i\rangle_{\mathsf F}
\bigg\|_{L_q}
\asymp_p
\begin{cases}
\rho_{p,G}(\Sigma)\mathcal L_1(N,q)
+
\|\Sigma\|^{p/2}\mathcal L_{p/2}(N,q),
& p \text{ even},\\[3pt]
\rho_{p,G}(\Sigma)\mathcal L_{1/2}(N,q)
+
\|\Sigma\|^{p/2}\mathcal L_{p/2}(N,q),
& p \text{ odd},
\end{cases}
\end{align*}
where
\[
\rho_{p,G}(\Sigma)
=
\begin{cases}
\|\Sigma\|_{\mathsf{F}}^{p/2-1}\|\Sigma\|, & p \text{ even},\\[3pt]
\|\Sigma\|_{\mathsf{F}}^{(p-1)/2}\|\Sigma\|^{1/2}, & p \text{ odd}.
\end{cases}
\]
Since \(\|\Sigma\|^{p/2}\le\rho_{p,G}(\Sigma)\),
\eqref{eq:elementary-latala-optimization} and
\eqref{eq:elementary-latala-optimization-power} give
\begin{align}
\sup_{\|A\|_{\mathsf F}=1}
\bigg\|
\sum_{i=1}^N\langle A,Z_i\rangle_{\mathsf F}
\bigg\|_{L_q}
\asymp_p
\begin{cases}
\rho_{p,G}(\Sigma)\left(\sqrt{Nq}+q\right)
+
\|\Sigma\|^{p/2}q^{p/2},
& p \text{ even},\\[3pt]
\rho_{p,G}(\Sigma)\sqrt{Nq}
+
\|\Sigma\|^{p/2}q^{p/2},
& p \text{ odd}.
\end{cases}
\label{eq:gaussian-scalar-sum}
\end{align}

It remains to consider \(1\le q<2\). The sub-Gaussian and Gaussian
upper bounds follow from moment monotonicity and the corresponding
estimates at \(q=2\), since their right-hand sides at \(q=2\) are
bounded, up to constants depending only on \(p\), by those at \(q\).

For the Gaussian lower bound, for every \(A\), the random variable
\(\sum_{i=1}^N \langle A,Z_i\rangle_{\mathsf F}\) is a polynomial of
degree at most \(p\) in Gaussian random variables. The finite-chaos
moment comparison
\cite[Theorem~5.10]{Janson1997GaussianHilbertSpaces} therefore gives
\[
\bigg\|
\sum_{i=1}^N \langle A,Z_i\rangle_{\mathsf F}
\bigg\|_{L_2}
\lesssim_p
\bigg\|
\sum_{i=1}^N \langle A,Z_i\rangle_{\mathsf F}
\bigg\|_{L_1}
\le
\bigg\|
\sum_{i=1}^N \langle A,Z_i\rangle_{\mathsf F}
\bigg\|_{L_q}.
\]
Taking the supremum over \(\|A\|_{\mathsf F}=1\), applying the Gaussian
lower bound at \(q=2\), and comparing its right-hand side with that at
\(q\) proves the desired lower bound for \(1\le q<2\). This completes
the proof.
\end{proof}

\subsection{Sharpness and the parity effect}\label{subsec:sharpness}

In this subsection, we show the sharpness of the sub-Gaussian bound
in Theorem~\ref{thm:main1}, and explain why the sub-Gaussian and Gaussian
intermediate covariance scales coincide at odd orders but can differ at
even orders. The following example shows that, for every fixed integer
\(p\ge2\), all three terms in the sub-Gaussian bound are necessary, up to
constants depending only on \(p\), and reveals the mechanism behind this
parity effect.

The construction combines two independent blocks. A Gaussian block supplies
the trace and large-deviation scales, as well as the intermediate scale at
odd orders. At even orders, its intermediate contribution is only
\(\rho_{p,G}\); the larger scale \(\rho_p\) is instead produced by a
random-radius Rademacher block, whose common bounded radius synchronizes
fluctuations across the eigendirections. This block is an anisotropic version
of the construction in \cite[Proposition~2.3]{puchkin2025sharper}.

\begin{example}
\label{example:subgaussian_tensor_sharpness}
Fix an integer \(p\ge2\). Let
\[
\lambda_1\ge\cdots\ge\lambda_d>0,
\qquad
\Lambda:=\operatorname{diag}(\lambda_1,\ldots,\lambda_d),
\qquad
\lambda_1=\|\Lambda\|.
\]
Let \(G\sim\mcN(0,\Lambda)\), let
\(\varepsilon\in\{\pm1\}^d\) have independent Rademacher entries, and let
\(R\), independent of \(G\) and \(\varepsilon\), take the values \(0\) and
\(\sqrt2\), each with probability \(1/2\). Define
\[
Y:=R\Lambda^{1/2}\varepsilon,
\qquad
X:=(G,Y)\in\mathbb R^{2d}.
\]
Standard Gaussian and Rademacher sub-Gaussian estimates, together with the
boundedness of \(R\), show that \(X\) satisfies \eqref{eq:def-subg} with
\(K\asymp1\). Since \(\E R^2=1\), \(X\) is centered and has covariance matrix
\[
\Sigma
=
\begin{bmatrix}
\Lambda&0\\
0&\Lambda
\end{bmatrix}.
\]
Consequently,
\begin{equation}\label{eq:subgaussian_tensor_covariance_relation}
\tr(\Sigma)=2\tr(\Lambda),
\qquad
\|\Sigma\|_{\mathsf F}=\sqrt2 \|\Lambda\|_{\mathsf F},
\qquad
\|\Sigma\|=\|\Lambda\|.
\end{equation}

For \(i=1,\ldots,N\), let \((G_i,R_i,\varepsilon_i)\) be independent
copies of \((G,R,\varepsilon)\), and set
\[
Y_i:=R_i\Lambda^{1/2}\varepsilon_i,
\qquad
X_i:=(G_i,Y_i).
\]
Then \(X_1,\ldots,X_N\) are i.i.d. with the same distribution as \(X\). Let
\[
S_p
:=
\frac1N\sum_{i=1}^N X_i^{\otimes p}-\E X^{\otimes p}.
\]
We show that, for every \(q\ge1\),
\begin{equation}\label{eq:subgaussian_tensor_sharpness}
\left(\E\|S_p\|_{\mathsf F}^q\right)^{1/q}
\gtrsim_p
\frac{\tr(\Sigma)^{p/2}}{\sqrt N}
+
\rho_p(\Sigma)\sqrt{\frac{q\land N}{N}}
+
\|\Sigma\|^{p/2}\frac{q^{p/2}}{N},
\end{equation}
where \(\rho_p(\Sigma)\) is defined in Theorem \ref{thm:main1}.

Restricting \(S_p\) to the tensor block corresponding entirely to the
\(G\)-coordinates, we obtain
\[
\|S_p\|_{\mathsf F}
\ge
\bigg\|
\frac1N\sum_{i=1}^N G_i^{\otimes p}-\E G^{\otimes p}
\bigg\|_{\mathsf F}.
\]
The Gaussian lower bound in Theorem~\ref{thm:main1} therefore yields
\begin{equation}\label{eq:subgaussian_tensor_gaussian_lower}
\left(\E\|S_p\|_{\mathsf F}^q\right)^{1/q}
\gtrsim_p
\frac{\tr(\Lambda)^{p/2}}{\sqrt N}
+
\rho_{p,G}(\Lambda)\sqrt{\frac qN}
+
\|\Lambda\|^{p/2}\frac{q^{p/2}}{N}.
\end{equation}

Suppose first that \(p\) is odd. In this case,
\[
\rho_{p,G}(\Lambda)
=
\|\Lambda\|_{\mathsf F}^{(p-1)/2}\|\Lambda\|^{1/2}
=
\rho_p(\Lambda).
\]
Moreover, $\sqrt{\frac qN}
\ge
\sqrt{\frac{q\land N}{N}}$.
Thus \eqref{eq:subgaussian_tensor_sharpness} follows from
\eqref{eq:subgaussian_tensor_gaussian_lower} and
\eqref{eq:subgaussian_tensor_covariance_relation}.

Now suppose that \(p=2k\) is even. Retaining the first and third terms in
\eqref{eq:subgaussian_tensor_gaussian_lower}, we have
\begin{equation}\label{eq:subgaussian_tensor_gaussian_scales}
\left(\E\|S_p\|_{\mathsf F}^q\right)^{1/q}
\gtrsim_p
\frac{\tr(\Lambda)^{p/2}}{\sqrt N}
+
\|\Lambda\|^{p/2}\frac{q^{p/2}}{N}.
\end{equation}
It remains to obtain the intermediate fluctuation scale lower bound. Set
\[
B:=\frac{\Lambda}{\|\Lambda\|_{\mathsf F}},
\qquad
A:=B^{\otimes k}.
\]
Then
\[
\|A\|_{\mathsf F}
=
\|B\|_{\mathsf F}^k
=
1.
\]
Since $\frac1N\sum_{i=1}^N Y_i^{\otimes p}-\E Y^{\otimes p}$
is the tensor block of \(S_p\) corresponding entirely to the
\(Y\)-coordinates, we have
\begin{equation}\label{eq:subgaussian_tensor_block_test}
\|S_p\|_{\mathsf F}
\ge
\bigg\|
\frac1N\sum_{i=1}^N Y_i^{\otimes p}-\E Y^{\otimes p}
\bigg\|_{\mathsf F}
\ge
\bigg|
\bigg\langle
A,
\frac1N\sum_{i=1}^N Y_i^{\otimes p}-\E Y^{\otimes p}
\bigg\rangle_{\mathsf F}
\bigg|.
\end{equation}

Since \(Y=R\Lambda^{1/2}\varepsilon\) and every coordinate of \(\varepsilon\) takes values in \(\{\pm 1\}\),
\[
Y^\top B Y
=
\frac{R^2}{\|\Lambda\|_{\mathsf F}}
\varepsilon^\top\Lambda^2\varepsilon
=
\frac{R^2}{\|\Lambda\|_{\mathsf F}}
\sum_{j=1}^d\lambda_j^2
=
R^2\|\Lambda\|_{\mathsf F}.
\]
Consequently, recalling that \(p=2k\),
\[
\left\langle A,Y^{\otimes p}\right\rangle_{\mathsf F}
=
\left\langle B^{\otimes k},Y^{\otimes 2k}\right\rangle_{\mathsf F}
=
(Y^\top B Y)^k
=
R^{2k}\|\Lambda\|_{\mathsf F}^k.
\]
For \( i = 1, \ldots, N\), define the Rademacher random variable \(\eta_i\) by 
\[
R_i^{2k}-\E R^{2k}
=
2^{k-1}\eta_i.
\]
Then \(\eta_1,\ldots,\eta_N\) are independent, and
\[
\bigg\langle
A,
\frac1N\sum_{i=1}^N Y_i^{\otimes p}-\E Y^{\otimes p}
\bigg\rangle_{\mathsf F}
=
2^{k-1}\|\Lambda\|_{\mathsf F}^k
\frac1N\sum_{i=1}^N\eta_i.
\]
Combining this identity with
\eqref{eq:subgaussian_tensor_block_test} and taking \(L_q\)-norms, we obtain
\begin{align}
\left(\E\|S_p\|_{\mathsf F}^q\right)^{1/q}
\ge
2^{k-1}\|\Lambda\|_{\mathsf F}^k
\bigg\|
\frac1N\sum_{i=1}^N\eta_i
\bigg\|_{L_q}
\asymp_p
\|\Lambda\|_{\mathsf F}^{p/2}
\sqrt{\frac{q\land N}{N}},
\label{eq:subgaussian_tensor_middle_even}
\end{align}
where the last step follows from the two-sided moment estimate for
Rademacher sums
\cite{montgomerysmith1990rademacher}.
Combining \eqref{eq:subgaussian_tensor_gaussian_scales} and
\eqref{eq:subgaussian_tensor_middle_even}, and using
\eqref{eq:subgaussian_tensor_covariance_relation}, proves
\eqref{eq:subgaussian_tensor_sharpness} in the even case.

\end{example}

\paragraph{Why the sub-Gaussian and Gaussian bounds differ}

The distinction is already visible at order \(p=2\). At this order, the
intermediate covariance scale is determined by the largest
\(L_2\)-fluctuation of a centered quadratic form. Since
\(G^{\otimes2}-\Lambda\) is symmetric, it suffices to consider symmetric
matrices \(A\). The Gaussian fourth-moment identity gives
\[
\E
\left\langle
A,G^{\otimes2}-\Lambda
\right\rangle_{\mathsf F}^{2}
=
2\|\Lambda^{1/2}A\Lambda^{1/2}\|_{\mathsf F}^{2}
\le
2\|\Lambda\|^2\|A\|_{\mathsf F}^2.
\]
Taking \(A\) to be the orthogonal projector onto a top eigendirection of
\(\Lambda\) gives equality. Therefore,
\[
\sup_{\|A\|_{\mathsf F}\le1}
\left\|
\left\langle
A,G^{\otimes2}-\Lambda
\right\rangle_{\mathsf F}
\right\|_{L_2}
=
\sqrt2\,\|\Lambda\|.
\]

Condition~\eqref{eq:def-subg} controls linear functionals, but places no
comparable restriction on the dependence among quadratic fluctuations. For
the random-radius block \(Y\) above, take
\(B=\Lambda/\|\Lambda\|_{\mathsf F}\). Then
\[
\left\langle
B,Y^{\otimes2}-\Lambda
\right\rangle_{\mathsf F}
=
(R^2-1)\|\Lambda\|_{\mathsf F},
\]
whose \(L_2\)-norm is \(\|\Lambda\|_{\mathsf F}\). Thus \(G\) and \(Y\) have
the same covariance and both satisfy \eqref{eq:def-subg} with constants of
order one, but their quadratic fluctuations can occur on different scales.
Indeed, in the eigenbasis of \(\Lambda\),
\(Y_j^2-\E Y_j^2=(R^2-1)\lambda_j\) for every \(j\). The same random factor
therefore makes the energy fluctuations in all eigendirections have the same
sign, so they add without cancellation. The sub-Gaussian control of linear
projections does not rule out this dependence. For \(G\), by contrast, the
coordinate energies are independent, and the fourth-moment identity yields
the spectral scale \(\|\Lambda\|\).

For higher orders, the same distinction appears through the Gaussian chaos
decomposition and the contractions used above. When \(p=2k\), centering
removes the zeroth Gaussian chaos, which is the sum of all complete Wick
pairings. The lowest nonzero chaos has order two and is formed from \(k-1\)
covariance factors and one centered quadratic Gaussian factor. Its largest
\(L_2\)-fluctuation has, up to constants depending only on \(p\), the scale
\(\|\Lambda\|_{\mathsf F}^{k-1}\|\Lambda\|\), and the higher Gaussian
chaoses are no larger. For the random-radius construction, by contrast, the
contraction against \(B^{\otimes k}\) remains random after centering and has
scale \(\|\Lambda\|_{\mathsf F}^{k}\). When \(p=2k+1\), a complete pairing
is impossible. The lowest Gaussian chaos has order one, while the analogous
contraction for the random-radius block also retains one linear factor. The
largest \(L_2\)-scale of this factor is \(\|\Lambda\|^{1/2}\), giving
\(\|\Lambda\|_{\mathsf F}^{k}\|\Lambda\|^{1/2}\) in both cases. Thus the
quadratic-fluctuation gap visible at order two persists at even orders,
whereas at odd orders the unavoidable linear factor gives the same covariance
scale.

\section{Polynomial moments under Gaussian convex domination}\label{sec:proof2}

This section proves our second main result, Theorem~\ref{thm:main2}, together
with the two propositions used in its proof. Theorem~\ref{thm:main2} is proved in
Subsection~\ref{subsec:proof2}, and
Propositions~\ref{prop:Hilbert-Gaussian-tensor-Lq}
and~\ref{prop:ULC-moment-tensors} in
Subsections~\ref{subsec:Hilbert-Gaussian-tensor-Lq}
and~\ref{subsec:ULC}, respectively.

\subsection{Proof of Theorem~\ref{thm:main2}}\label{subsec:proof2}

The proof follows the argument outlined in Subsection~\ref{subsec:main2}.
We first reduce to the case where \(Z\) is finitely supported with a
strict slack in the convex domination. We then apply
\cite[Proposition~3.3]{hua2026talagrand} to obtain a structured martingale
coupling of \(Z\) and \(G_\Sigma\), whose remainder \(G_\Sigma-Z\) has
uniformly log-concave conditional law, so that
Proposition~\ref{prop:ULC-moment-tensors} bounds the conditional moment
tensors. Next, we derive the conditional expansion and control its
Gaussian contribution via
Proposition~\ref{prop:Hilbert-Gaussian-tensor-Lq}. Finally, we bound the
lower-degree contributions using the induction hypothesis together with
these bounds, which closes the induction on the degree.

\begin{proof}[Proof of Theorem~\ref{thm:main2}]

\medskip
\noindent\textbf{Step 1: Reduction to finite support and strict slack.}
\medskip

Set $E:=\operatorname{Ran}(\Sigma)$.
For every \(v\in\ker(\Sigma)\), convex domination applied to
\(x\mapsto|\langle v,x\rangle|\) gives
\[
\E|\langle v,Z\rangle|
\le
\E|\langle v,G_\Sigma\rangle|
=
0.
\]
Thus \(Z\in E\) almost surely. Restricting \(\Sigma\) and every input leg
of \(T\) to \(E\) does not change \(T(Z,\ldots,Z)\) or
\(\Psi_p(q,\Sigma)\), and cannot increase \(\|T\|_{\mathsf F}\). Moreover, since $E=\ker(\Sigma)^\perp$,
the restriction of \(\Sigma\) to \(E\) is positive definite. After
identifying \(E\) with \(\mathbb R^{\operatorname{rank}(\Sigma)}\), we
may therefore assume that \(\Sigma\) is positive definite.

Let \(G\sim\mathcal N(0,I_d)\). Since convex order is preserved under
linear maps, $\Sigma^{-1/2}Z
\preceq_{\mathrm{cx}}
G$. Applying \cite[Lemma~3.1]{hua2026talagrand} to
\(\Sigma^{-1/2}Z\), we obtain centered finitely supported random vectors
\(Y_n\) such that $Y_n\preceq_{\mathrm{cx}}G$
and \(Y_n\) converges in distribution to \(\Sigma^{-1/2}Z\). Define $Z_n:=\Sigma^{1/2}Y_n.$
It follows that
\[
Z_n\preceq_{\mathrm{cx}}G_\Sigma
\qquad\text{and}\qquad
Z_n\Longrightarrow Z.
\]
Choose \(\varepsilon_n\downarrow0\) with \(0<\varepsilon_n<1\), and set
\(\widetilde Z_n:=(1-\varepsilon_n)Z_n\). Since convex order is preserved
under linear maps,
\[
\widetilde Z_n\preceq_{\mathrm{cx}}(1-\varepsilon_n)G_\Sigma
\qquad\text{and}\qquad
\widetilde Z_n\Longrightarrow Z.
\]

It therefore suffices to prove the estimate, with a constant \(C_p\)
depending only on \(p\) and in particular independent of \(\varepsilon\)
and of the distribution of \(W\), for every centered finitely supported
random vector \(W\) satisfying
\(W\preceq_{\mathrm{cx}}(1-\varepsilon)G_\Sigma\) with
\(0<\varepsilon<1\). Indeed, applying this estimate to
\(W=\widetilde Z_n\), and using that \(x\mapsto T(x,\ldots,x)\) is
continuous together with the lower semicontinuity of nonnegative moments
under convergence in distribution, we obtain
\[
\E\big\|T(Z,\ldots,Z)\big\|_{\mathcal H}^q
\le
\liminf_{n\to\infty}
\E\big\|T(\widetilde Z_n,\ldots,\widetilde Z_n)\big\|_{\mathcal H}^q
\le
C_p^q\Psi_p(q,\Sigma)^q\|T\|_{\mathsf F}^q.
\]
Consequently, it suffices to prove the result when
\begin{equation}
Z \text{ has finite support},
\qquad
Z\preceq_{\mathrm{cx}}(1-\varepsilon)G_\Sigma
\text{ for some } \varepsilon\in(0,1).
\label{eq:Lq-strict-slack}
\end{equation}

\medskip
\noindent\textbf{Step 2: Structured martingale coupling and conditional moments.}
\medskip

We now work under the additional assumption
\eqref{eq:Lq-strict-slack}. Write
\[
\supp(Z)=\{z_1,\ldots,z_s\},
\qquad
\mathbb P(Z=z_i)=w_i>0.
\]
By the strict-slack part of \cite[Proposition~3.3]{hua2026talagrand},
\(Z\) and \(G_\Sigma\) admit a martingale coupling whose
conditional densities are exponential tilts: the law of \(G_\Sigma\) given \(Z=z_i\) has density \(f_i\)
with respect to \(\gamma_\Sigma:=\mcN(0,\Sigma)\) given by
\[
f_i(y):=
\frac{\exp\bigl(a_i+\langle b_i,y\rangle\bigr)}{D(y)}, \qquad
D(y):=
\sum_{j=1}^s
w_j\exp\bigl(a_j+\langle b_j,y\rangle\bigr),
\]
 for suitable \(a_1,\ldots,a_s\in\mathbb R\) and
\(b_1,\ldots,b_s\in\mathbb R^d\).  Moreover,
\begin{align}\label{eq:theorem2_proof_aux1}
\int_{\mathbb R^d}f_i(y)\,\gamma_\Sigma(dy)=1,
\qquad
\int_{\mathbb R^d}y f_i(y)\,\gamma_\Sigma(dy)=z_i,
\end{align}
the second identity expressing the martingale property
\(\E[G_\Sigma\mid Z]=Z\). Set
\[
R:=G_\Sigma-Z,
\qquad
M_r(z_i):=\E[R^{\otimes r}\mid Z=z_i],\ \ r\ge 1.
\]
In particular, the second identity in \eqref{eq:theorem2_proof_aux1} gives
\(M_1(z_i)=0\).

Conditionally on \(Z=z_i\), the random vector \(R\) has density
proportional to
\[
y\longmapsto
\exp\bigl(-\mathcal V_i(y)\bigr),
\]
where
\[
\mathcal V_i(y)
:=
\frac12
\left\langle
y+z_i,\Sigma^{-1}(y+z_i)
\right\rangle
-\log f_i(y+z_i).
\]
Since \(f_i\) is a strictly positive finite softmax function,
\(\mathcal V_i\in C^\infty(\mathbb R^d)\). Furthermore, defining
\[
\alpha_j(y)
:=
\frac{
w_j\exp\bigl(a_j+\langle b_j,y\rangle\bigr)
}{
D(y)
},
\qquad
\overline b(y):=
\sum_{j=1}^s\alpha_j(y)b_j,
\]
we have
\begin{align*}
\nabla^2\mathcal V_i(y)
&=
\Sigma^{-1}
+
\nabla^2\log D(y+z_i)\\
&=
\Sigma^{-1}
+
\sum_{j=1}^s
\alpha_j(y+z_i)
\bigl(b_j-\overline b(y+z_i)\bigr)
\bigl(b_j-\overline b(y+z_i)\bigr)^\top\\
&\succeq
\Sigma^{-1}.
\end{align*}
Thus the conditional law of \(R\) given \(Z=z_i\) satisfies the assumptions
of Proposition~\ref{prop:ULC-moment-tensors}, which gives
\begin{equation}
\|M_r(z_i)\|_{\mathsf F}
\le
U_r\rho_r(\Sigma),
\qquad 1\le i\le s,
\label{eq:Lq-conditional-Mr}
\end{equation}
where \(U_r\) is a constant depending only on \(r\), and \(\rho_r(\Sigma)\)
is as in the proposition.

\medskip
\noindent\textbf{Step 3: Conditional expansion and the Gaussian term.}
\medskip

In what follows, let \(e_1,\ldots,e_d\) denote the standard basis vectors of \(\mathbb R^d\). For \(0\le r\le p\) and
\(B\in(\mathbb R^d)^{\otimes r}\), write
\[
B
=
\sum_{i_1,\ldots,i_r=1}^d
B_{i_1,\ldots,i_r}
e_{i_1}\otimes\cdots\otimes e_{i_r},
\]
and define
\[
T[Z^{p-r},B]
:=
\sum_{i_1,\ldots,i_r=1}^d
B_{i_1,\ldots,i_r}
T(
\underbrace{Z,\ldots,Z}_{p-r\text{ times}},
e_{i_1},\ldots,e_{i_r}
).
\]
When \(r=p\), the \(Z\)-arguments are omitted; when $r=0$, $T[Z^p]
=
T(Z,\ldots,Z)$.

Since symmetrizing \(T\) does not change \(T(Z,\ldots,Z)\) and cannot
increase \(\|T\|_{\mathsf F}\), we may assume without loss of generality
that \(T\) is symmetric.

We prove the desired estimate simultaneously for all \(q\ge1\) and all
finite-dimensional target Hilbert spaces by induction on \(p\). For
\(p=1\), the function
\[
x\longmapsto\|T(x)\|_{\mathcal H}^q
\]
is convex. Hence the assumed convex domination and
Proposition~\ref{prop:Hilbert-Gaussian-tensor-Lq} give
\[
\|T(Z)\|_{L_q(\mathcal H)}
\le
\|T(G_\Sigma)\|_{L_q(\mathcal H)}
\lesssim
\Psi_1(q,\Sigma)\|T\|_{\mathsf F}.
\]

Let \(p\ge2\), and suppose that the assertion has been proved through
degree \(p-1\). Expanding \(T(G_\Sigma,\ldots,G_\Sigma)=T(Z+R,\ldots,Z+R)\)
by multilinearity produces, for each \(0\le r\le p\), the \(\binom pr\)
terms in which \(r\) arguments equal \(R\) and the remaining \(p-r\) equal
\(Z\). Since \(T\) is symmetric, all terms with the same \(r\) coincide,
and therefore
\[
T(G_\Sigma,\ldots,G_\Sigma)
=
\sum_{r=0}^p
\binom pr
T[Z^{p-r},R^{\otimes r}].
\]

Taking conditional expectation given \(Z\), we obtain
\begin{align*}
\E\left[
T[Z^{p-r},R^{\otimes r}]
\,\middle|\,
Z
\right]
=
T\left[
Z^{p-r},
\E[R^{\otimes r}\mid Z]
\right]=
T[Z^{p-r},M_r(Z)].
\end{align*}
The term corresponding to \(r=0\) is \(T[Z^p]\), while the term
corresponding to \(r=1\) vanishes because $M_1(Z)
=
\E[R\mid Z]
=
0$.
Consequently,
\begin{align*}
\E\left[
T(G_\Sigma,\ldots,G_\Sigma)
\,\middle|\,
Z
\right]
=
T[Z^p]
+
\sum_{r=2}^p
\binom pr
T[Z^{p-r},M_r(Z)].
\end{align*}
Rearranging this identity, we obtain
\begin{align}
T[Z^p]
=
\E\left[
T(G_\Sigma,\ldots,G_\Sigma)
\,\middle|\,
Z
\right]
-
\sum_{r=2}^p
\binom pr
T[Z^{p-r},M_r(Z)].
\label{eq:Lq-conditional-expansion}
\end{align}

We next control $\E\left[
T(G_\Sigma,\ldots,G_\Sigma)
\,\middle|\,
Z
\right]$. Since
\(y\mapsto\|y\|_{\mathcal H}^q\) is convex for \(q\ge1\), conditional
Jensen's inequality gives
\begin{align*}
\E
\left\|
\E\left[
T(G_\Sigma,\ldots,G_\Sigma)
\,\middle|\,
Z
\right]
\right\|_{\mathcal H}^q\le
\E
\left[
\E\left[
\big\|
T(G_\Sigma,\ldots,G_\Sigma)
\big\|_{\mathcal H}^q
\,\middle|\,
Z
\right]
\right]=
\E
\big\|
T(G_\Sigma,\ldots,G_\Sigma)
\big\|_{\mathcal H}^q.
\end{align*}
Taking \(q\)-th roots and applying Proposition~\ref{prop:Hilbert-Gaussian-tensor-Lq} yield
\begin{align}
\left\|
\E\left[
T(G_\Sigma,\ldots,G_\Sigma)
\,\middle|\,
Z
\right]
\right\|_{L_q(\mathcal H)}
\le
\left\|
T(G_\Sigma,\ldots,G_\Sigma)
\right\|_{L_q(\mathcal H)}
\le g_p
\Psi_p(q,\Sigma)
\|T\|_{\mathsf F},
\label{eq:Lq-conditional-Gaussian}
\end{align}
where \(g_p\) depends only on \(p\).

\medskip
\noindent\textbf{Step 4: Lower-order terms and completion of the induction.}
\medskip

For \(2\le r\le p-1\), we equip $\mathcal H\otimes(\mathbb R^d)^{\otimes r}$ with its natural Hilbert-space structure. More precisely, if
\[
U
=
\sum_{i_1,\ldots,i_r=1}^d
h_{i_1,\ldots,i_r}
\otimes
e_{i_1}\otimes\cdots\otimes e_{i_r},
\qquad
h_{i_1,\ldots,i_r}\in\mathcal H,
\]
then
\[
\|U\|_{\mathcal H\otimes(\mathbb R^d)^{\otimes r}}^2
:=
\sum_{i_1,\ldots,i_r=1}^d
\|h_{i_1,\ldots,i_r}\|_{\mathcal H}^2.
\]

Define the \((p-r)\)-linear map
\[
\mathcal S_rT:
(\mathbb R^d)^{p-r}
\longrightarrow
\mathcal H\otimes(\mathbb R^d)^{\otimes r}
\]
by
\begin{align*}
(\mathcal S_rT)(x_1,\ldots,x_{p-r})
:=
\sum_{i_1,\ldots,i_r=1}^d
T(x_1,\ldots,x_{p-r},e_{i_1},\ldots,e_{i_r})\otimes
e_{i_1}\otimes\cdots\otimes e_{i_r}.
\end{align*}
By the definition of the Frobenius norm of a multilinear map,
\begin{align*}
\|\mathcal S_rT\|_{\mathsf F}^2
&=
\sum_{j_1,\ldots,j_{p-r}=1}^d
\big\|
(\mathcal S_rT)
(e_{j_1},\ldots,e_{j_{p-r}})
\big\|_{\mathcal H\otimes(\mathbb R^d)^{\otimes r}}^2\\
&=
\sum_{\substack{
j_1,\ldots,j_{p-r}=1\\
i_1,\ldots,i_r=1
}}^d
\big\|
T(
e_{j_1},\ldots,e_{j_{p-r}},
e_{i_1},\ldots,e_{i_r}
)
\big\|_{\mathcal H}^2=
\|T\|_{\mathsf F}^2.
\end{align*}
Therefore,
\[
\|\mathcal S_rT\|_{\mathsf F}
=
\|T\|_{\mathsf F}.
\]

We next relate \(\mathcal S_rT\) to the contraction appearing in
\eqref{eq:Lq-conditional-expansion}. By definition,
\begin{align*}
T[Z^{p-r},M_r(Z)]
=
\sum_{i_1,\ldots,i_r=1}^d
(M_r(Z))_{i_1,\ldots,i_r}\,
T(
\underbrace{Z,\ldots,Z}_{p-r\text{ times}},
e_{i_1},\ldots,e_{i_r}
).
\end{align*}
Hence, by the triangle inequality followed by the Cauchy--Schwarz
inequality,
\begin{align*}
\big\|
T[Z^{p-r},M_r(Z)]
\big\|_{\mathcal H}
&\le
\sum_{i_1,\ldots,i_r=1}^d
\big|
(M_r(Z))_{i_1,\ldots,i_r}
\big|\times
\big\|
T(
Z,\ldots,Z,
e_{i_1},\ldots,e_{i_r}
)
\big\|_{\mathcal H}\\
&\le
\bigg(
\sum_{i_1,\ldots,i_r=1}^d
(M_r(Z))_{i_1,\ldots,i_r}^2
\bigg)^{1/2}
\bigg(
\sum_{i_1,\ldots,i_r=1}^d
\big\|
T(
Z,\ldots,Z,
e_{i_1},\ldots,e_{i_r}
)
\big\|_{\mathcal H}^2
\bigg)^{1/2}\\
&=
\|M_r(Z)\|_{\mathsf F}
\big\|
(\mathcal S_rT)(Z,\ldots,Z)
\big\|_{\mathcal H\otimes(\mathbb R^d)^{\otimes r}}.
\end{align*}
Thus, pointwise,
\[
\big\|
T[Z^{p-r},M_r(Z)]
\big\|_{\mathcal H}
\le
\big\|
(\mathcal S_rT)(Z,\ldots,Z)
\big\|_{\mathcal H\otimes(\mathbb R^d)^{\otimes r}}
\|M_r(Z)\|_{\mathsf F}.
\]
By \eqref{eq:Lq-conditional-Mr},
\[
\|M_r(Z)\|_{\mathsf F}
\le
U_r\rho_r(\Sigma)
\]
almost surely. Therefore, taking \(L_q\)-norms in the preceding
pointwise inequality gives
\begin{align}
\big\|
T[Z^{p-r},M_r(Z)]
\big\|_{L_q(\mathcal H)}
\le
U_r\rho_r(\Sigma)
\big\|
(\mathcal S_rT)(Z,\ldots,Z)
\big\|_{L_q(
\mathcal H\otimes(\mathbb R^d)^{\otimes r}
)}.
\label{eq:Lq-correction-first-bound}
\end{align}

Let \(Q_k\) denote the constant in the induction hypothesis at degree
\(k\). More precisely, the induction hypothesis states that, for every
finite-dimensional Hilbert space \(\mathcal K\) and every \(k\)-linear
map
\[
S:(\mathbb R^d)^k\to\mathcal K,
\]
one has
\[
\big\|
S(Z,\ldots,Z)
\big\|_{L_q(\mathcal K)}
\le
Q_k
\Psi_k(q,\Sigma)
\|S\|_{\mathsf F}.
\]
The space $\mathcal H\otimes(\mathbb R^d)^{\otimes r}$ is a finite-dimensional Hilbert space, and \(\mathcal S_rT\) is an
\((p-r)\)-linear map taking values in this space. We may therefore apply
the induction hypothesis at degree \(p-r\) to \(\mathcal S_rT\). This
gives
\begin{align*}
\big\|
(\mathcal S_rT)(Z,\ldots,Z)
\big\|_{L_q(
\mathcal H\otimes(\mathbb R^d)^{\otimes r}
)}
\le
Q_{p-r}
\Psi_{p-r}(q,\Sigma)
\|\mathcal S_rT\|_{\mathsf F}=
Q_{p-r}
\Psi_{p-r}(q,\Sigma)
\|T\|_{\mathsf F}.
\end{align*}
Substituting this estimate into
\eqref{eq:Lq-correction-first-bound}, we obtain
\begin{align}
\big\|
T[Z^{p-r},M_r(Z)]
\big\|_{L_q(\mathcal H)}
\le
U_rQ_{p-r}
\rho_r(\Sigma)
\Psi_{p-r}(q,\Sigma)
\|T\|_{\mathsf F}.
\label{eq:Lq-correction-before-scales}
\end{align}

It remains to compare the two covariance scales. We claim that
\begin{equation}
\rho_r(\Sigma)
\Psi_{p-r}(q,\Sigma)
\le
\Psi_p(q,\Sigma),
\qquad
2\le r\le p-1.
\label{eq:rho-Psi-product}
\end{equation}
Suppose first that \(r\) is even. Then $\rho_r(\Sigma)
=
\|\Sigma\|_{\mathsf F}^{r/2}$,
and hence
\begin{align*}
\rho_r(\Sigma)\Psi_{p-r}(q,\Sigma)
&=
\sum_{\substack{
0\le j\le p-r\\
j\equiv p-r\!\!\!\!\pmod 2
}}
\|\Sigma\|_{\mathsf F}^{r/2}
\|\Sigma\|_{\mathsf F}^{(p-r-j)/2}
\|\Sigma\|^{j/2}q^{j/2}\\
&=
\sum_{\substack{
0\le j\le p-r\\
j\equiv p\!\!\!\!\pmod 2
}}
\|\Sigma\|_{\mathsf F}^{(p-j)/2}
\|\Sigma\|^{j/2}q^{j/2}\le
\Psi_p(q,\Sigma).
\end{align*}
Now suppose that \(r\) is odd. Then $\rho_r(\Sigma)
=
\|\Sigma\|_{\mathsf F}^{(r-1)/2}
\|\Sigma\|^{1/2}$, and
\begin{align*}
\rho_r(\Sigma)\Psi_{p-r}(q,\Sigma)
&=
\sum_{\substack{
0\le j\le p-r\\
j\equiv p-r\!\!\!\!\pmod 2
}}
\|\Sigma\|_{\mathsf F}^{(p-j-1)/2}
\|\Sigma\|^{(j+1)/2}q^{j/2}\\
&\le
\sum_{\substack{
0\le j\le p-r\\
j\equiv p-r\!\!\!\!\pmod 2
}}
\|\Sigma\|_{\mathsf F}^{(p-j-1)/2}
\|\Sigma\|^{(j+1)/2}q^{(j+1)/2}\le
\Psi_p(q,\Sigma),
\end{align*}
where the first inequality uses \(q\ge1\). For every index \(j\) in
the last sum, \(j+1\equiv p\pmod 2\), so its summand is one of the
summands defining \(\Psi_p(q,\Sigma)\). This proves
\eqref{eq:rho-Psi-product}.

Combining
\eqref{eq:Lq-correction-before-scales} and
\eqref{eq:rho-Psi-product}, we conclude that
\[
\big\|
T[Z^{p-r},M_r(Z)]
\big\|_{L_q(\mathcal H)}
\le
U_rQ_{p-r}
\Psi_p(q,\Sigma)
\|T\|_{\mathsf F},
\qquad
2\le r\le p-1.
\]

The term corresponding to \(r=p\) is estimated directly. By
\eqref{eq:Lq-conditional-Mr},
\begin{align*}
\big\|
T[M_p(Z)]
\big\|_{L_q(\mathcal H)}
\le
\|T\|_{\mathsf F}
\big\|
\|M_p(Z)\|_{\mathsf F}
\big\|_{L_q}\le
U_p\rho_p(\Sigma)\|T\|_{\mathsf F}\le
U_p\Psi_p(q,\Sigma)\|T\|_{\mathsf F}.
\end{align*}

Using the representation of \(T[Z^p]\) in \eqref{eq:Lq-conditional-expansion}, and combining
the preceding estimates with
\eqref{eq:Lq-conditional-Gaussian}, gives
\begin{align*}
\|T[Z^p]\|_{L_q(\mathcal H)}
\le
\bigg(
g_p
+
U_p
+
\sum_{r=2}^{p-1}
\binom pr
U_rQ_{p-r}
\bigg)
\Psi_p(q,\Sigma)
\|T\|_{\mathsf F}.
\end{align*}
Thus the induction closes with $Q_1:=g_1$
and, for \(p\ge2\),
\[
Q_p
:=
g_p
+
U_p
+
\sum_{r=2}^{p-1}
\binom pr
U_rQ_{p-r},
\]
where the sum is understood to be empty when \(p=2\). In particular,
\(Q_p\) depends only on \(p\) and is independent of \(q\), the dimension $d$,
\(\Sigma\), the support of \(Z\), and the strict-slack parameter
\(\varepsilon\).

This proves the desired estimate in the finite-support strict-slack
case. The approximation argument in \textbf{Step~1} completes the proof in full generality.
\end{proof}

\subsection{Polynomial moments for Gaussian vectors}\label{subsec:Hilbert-Gaussian-tensor-Lq}

Proposition~\ref{prop:Hilbert-Gaussian-tensor-Lq} bounds the \(L_q\)-moments
of a Hilbert-valued polynomial evaluated at a Gaussian vector, in both
uncentered and centered form. It is the Gaussian case of
Theorem~\ref{thm:main2}, and enters its proof through the Gaussian
contribution in \eqref{eq:Lq-conditional-expansion}; the centered form is what
gives the sharper Gaussian scale in
Lemma~\ref{lem:subgaussian-directional-tensor-moments}. We first decompose \(T(G_\Sigma,\ldots,G_\Sigma)\) into Gaussian chaos
components, then bound the moments of each component separately, and
finally sum the resulting estimates, which are exactly the summands
defining \(\Psi_p(q,\Sigma)\) and \(\Psi_p^\circ(q,\Sigma)\).

\begin{proof}[Proof of Proposition~\ref{prop:Hilbert-Gaussian-tensor-Lq}]
By orthogonal invariance, we may assume that
\[
\Sigma
=
\operatorname{diag}(\lambda_1,\ldots,\lambda_d),
\qquad
G_\Sigma
=
\left(
\lambda_1^{1/2}\xi_1,\ldots,\lambda_d^{1/2}\xi_d
\right),
\]
where \(\xi_1,\ldots,\xi_d\) are independent standard Gaussian random
variables; indeed, an orthogonal change of coordinates can be absorbed
into \(T\) without changing its Frobenius norm. In these coordinates,
\(\|\Sigma\|=\max_{1\le i\le d}\lambda_i\) and
\(\|\Sigma\|_{\mathsf F}^2=\sum_{i=1}^d\lambda_i^2\). Since symmetrizing \(T\) does not change
\(T(G_\Sigma,\ldots,G_\Sigma)\) and cannot increase
\(\|T\|_{\mathsf F}\), we may assume without loss of generality
that \(T\) is symmetric. Writing
\(T_{i_1,\ldots,i_p}:=T(e_{i_1},\ldots,e_{i_p})\in\mathcal H\), we have
\begin{align}
T(G_\Sigma,\ldots,G_\Sigma)
=
\sum_{i_1,\ldots,i_p=1}^d
\lambda_{i_1}^{1/2}\cdots\lambda_{i_p}^{1/2}
T_{i_1,\ldots,i_p}
\xi_{i_1}\cdots\xi_{i_p}.
\label{eq:Gaussian-coordinate-expansion}
\end{align}

\medskip
\noindent\textbf{Step 1: Chaos decomposition.}
\medskip

Let \(H_r\) denote the probabilists' Hermite polynomial of degree \(r\).
For an ordered \(j\)-tuple
\(\mathbf k=(k_1,\ldots,k_j)\in\{1,\ldots,d\}^j\), define
\[
\nu_\ell(\mathbf k)
:=
\bigl|\{r:k_r=\ell\}\bigr|,
\qquad
H_{\mathbf k}(\xi)
:=
\prod_{\ell=1}^d
H_{\nu_\ell(\mathbf k)}(\xi_\ell),
\]
with the convention \(H_{\varnothing}=1\). Since
\(\sum_{\ell=1}^d\nu_\ell(\mathbf k)=j\), the multivariate Hermite
polynomial \(H_{\mathbf k}(\xi)\) has total Hermite degree \(j\) and hence
belongs to the \(j\)-th homogeneous Gaussian chaos. When \(j\ge1\), it is
also centered. Indeed, at least one of the multiplicities
\(\nu_\ell(\mathbf k)\) is positive, so independence, together with
\(\E H_m(\xi_\ell)=0\) for \(m\ge1\), yields
\(\E H_{\mathbf k}(\xi)=0\).

For \(0\le a\le\lfloor p/2\rfloor\), let \(\mathcal P_a(p)\) be the
collection of all choices of \(a\) disjoint pairs of positions in
\(\{1,\ldots,p\}\), and let \(V(\pi)\) denote the set of positions covered
by \(\pi\in\mathcal P_a(p)\). The Wick expansion in coordinates reads
\begin{align}
\xi_{i_1}\cdots\xi_{i_p}
=
\sum_{a=0}^{\lfloor p/2\rfloor}
\sum_{\pi\in\mathcal P_a(p)}
\bigg(
\prod_{\{r,s\}\in\pi}
\mathbf 1_{\{i_r=i_s\}}
\bigg)
H_{(i_t)_{t\notin V(\pi)}}(\xi),
\label{eq:coordinate-Wick-expansion}
\end{align}
where each pair identifies two coordinate indices, while the remaining
\(p-2a\) indices form a Hermite polynomial of total degree \(p-2a\).

Fix \(a\) and set \(j:=p-2a\). Substituting
\eqref{eq:coordinate-Wick-expansion} into
\eqref{eq:Gaussian-coordinate-expansion}, the pairs contract \(T\) against
\(a\) copies of \(\Sigma\), which leads us to define
\[
B^{(a)}_{k_1,\ldots,k_j}
:=
\sum_{r_1,\ldots,r_a=1}^d
\lambda_{r_1}\cdots\lambda_{r_a}
T_{r_1,r_1,\ldots,r_a,r_a,k_1,\ldots,k_j},
\]
so that \(B^{(0)}_{k_1,\ldots,k_p}=T_{k_1,\ldots,k_p}\). Since \(T\) is
symmetric, all pairings in \(\mathcal P_a(p)\) give the same contribution,
and \(|\mathcal P_a(p)|=p!/(2^aa!j!)\). Consequently, setting
\[
Y_j
:=
\frac{p!}{2^aa!j!}
\sum_{k_1,\ldots,k_j=1}^d
\lambda_{k_1}^{1/2}\cdots\lambda_{k_j}^{1/2}
B^{(a)}_{k_1,\ldots,k_j}
H_{(k_1,\ldots,k_j)}(\xi),
\]
we obtain the chaos decomposition
\begin{equation}
T(G_\Sigma,\ldots,G_\Sigma)
=
\sum_{\substack{0\le j\le p\\j\equiv p\!\!\!\pmod{2}}}
Y_j,
\label{eq:Gaussian-coordinate-chaos-decomposition}
\end{equation}
where each \(Y_j\) is an \(\mathcal H\)-valued random variable in the
\(j\)-th Gaussian chaos generated by \(\xi_1,\ldots,\xi_d\). More
precisely, for every \(u\in\mathcal H\), the scalar random variable
\(\langle Y_j,u\rangle_{\mathcal H}\) is a linear combination of
multivariate Hermite polynomials of total Hermite degree \(j\).

\medskip
\noindent\textbf{Step 2: Moments of the chaos components.}
\medskip

The Cauchy--Schwarz inequality gives
\begin{align*}
\sum_{k_1,\ldots,k_j=1}^d
\big\|B^{(a)}_{k_1,\ldots,k_j}\big\|_{\mathcal H}^2
&\le
\bigg(
\sum_{r_1,\ldots,r_a=1}^d
\lambda_{r_1}^2\cdots\lambda_{r_a}^2
\bigg)
\sum_{k_1,\ldots,k_j=1}^d
\sum_{r_1,\ldots,r_a=1}^d
\big\|
T_{r_1,r_1,\ldots,r_a,r_a,k_1,\ldots,k_j}
\big\|_{\mathcal H}^2
\\
&\le
\|\Sigma\|_{\mathsf F}^{2a}\|T\|_{\mathsf F}^2.
\end{align*}
The Hermite orthogonality relations read
\[
\E
\left[
H_{(k_1,\ldots,k_j)}(\xi)
H_{(\ell_1,\ldots,\ell_j)}(\xi)
\right]
=
\sum_{\sigma\in S_j}
\prod_{r=1}^j
\mathbf{1}_{\{k_r=\ell_{\sigma(r)}\}},
\]
where \(S_j\) denotes the set of all permutations of
\(\{1,\ldots,j\}\). Since
\(B^{(a)}_{k_1,\ldots,k_j}\) is symmetric in
\(k_1,\ldots,k_j\), expanding
\(\E\|Y_j\|_{\mathcal H}^2\) and applying these relations gives
\[
\left(\E\|Y_j\|_{\mathcal H}^2\right)^{1/2}
=
\frac{p!}{2^a a! j!}\sqrt{j!}
\bigg(
\sum_{k_1,\ldots,k_j=1}^d
\lambda_{k_1}\cdots\lambda_{k_j}
\big\|B^{(a)}_{k_1,\ldots,k_j}\big\|_{\mathcal H}^2
\bigg)^{1/2}
\lesssim_p
\|\Sigma\|_{\mathsf F}^{a}
\|\Sigma\|^{j/2}
\|T\|_{\mathsf F}.
\]
Here we used the bound on \(B^{(a)}\) established above and
\(\lambda_k\le\|\Sigma\|\). The remaining prefactor is absorbed into
the implicit constant since \(a\le p/2\) and \(j\le p\).

Now let \(q\ge2\), set \(D:=\dim(\mathcal H)\), and let
\(u_1,\ldots,u_D\) be an orthonormal basis of \(\mathcal H\). Each scalar projection
\(\langle Y_j,u_r\rangle_{\mathcal H}\) belongs to the \(j\)-th homogeneous
Gaussian chaos. Therefore, Gaussian hypercontractivity
\cite[Theorem~5.10]{Janson1997GaussianHilbertSpaces} gives $\left\|
\langle Y_j,u_r\rangle_{\mathcal H}
\right\|_{L_q}
\le
(q-1)^{j/2}
\left\|
\langle Y_j,u_r\rangle_{\mathcal H}
\right\|_{L_2}$. Hence, by the triangle
inequality in \(L_{q/2}\),
\begin{align*}
\left(\E\|Y_j\|_{\mathcal H}^q\right)^{2/q}
&=
\bigg\|
\sum_{r=1}^D\big|\langle Y_j,u_r\rangle_{\mathcal H}\big|^2
\bigg\|_{L_{q/2}}
\le
\sum_{r=1}^D
\big\|\langle Y_j,u_r\rangle_{\mathcal H}\big\|_{L_q}^2
\le
(q-1)^j\,\E\|Y_j\|_{\mathcal H}^2,
\end{align*}
and combining the two displays with \(a=(p-j)/2\) gives
\begin{equation}
\left(\E\|Y_j\|_{\mathcal H}^q\right)^{1/q}
\lesssim_p
\|\Sigma\|_{\mathsf F}^{(p-j)/2}
\|\Sigma\|^{j/2}q^{j/2}
\|T\|_{\mathsf F}.
\label{eq:Gaussian-coordinate-chaos-Lq}
\end{equation}

\medskip
\noindent\textbf{Step 3: Completion of the proof.}
\medskip

Applying the triangle inequality to
\eqref{eq:Gaussian-coordinate-chaos-decomposition} and using
\eqref{eq:Gaussian-coordinate-chaos-Lq}, we obtain
\[
\left(\E\big\|T(G_\Sigma,\ldots,G_\Sigma)\big\|_{\mathcal H}^q\right)^{1/q}
\lesssim_p
\sum_{\substack{0\le j\le p\\j\equiv p\!\!\!\pmod{2}}}
\|\Sigma\|_{\mathsf F}^{(p-j)/2}
\|\Sigma\|^{j/2}q^{j/2}
\|T\|_{\mathsf F}
=
\Psi_p(q,\Sigma)\|T\|_{\mathsf F},
\]
which is the first estimate for \(q\ge2\). Every \(Y_j\) with \(j\ge1\) is
centered, and \(Y_0\), which is present only when \(p\) is even, is
deterministic; hence \(Y_0=\E T(G_\Sigma,\ldots,G_\Sigma)\) and
\[
T(G_\Sigma,\ldots,G_\Sigma)
-
\E T(G_\Sigma,\ldots,G_\Sigma)
=
\sum_{\substack{1\le j\le p\\j\equiv p\!\!\!\pmod{2}}}
Y_j.
\]
Repeating the preceding estimate over this restricted range gives the
second estimate for \(q\ge2\), with \(\Psi_p(q,\Sigma)\) replaced by
\(\Psi_p^\circ(q,\Sigma)\).

Finally, for \(1\le q<2\) both estimates follow from the case \(q=2\) by
monotonicity of \(L_q\)-norms, since
\(\Psi_p(2,\Sigma)\le2^{p/2}\Psi_p(q,\Sigma)\) and
\(\Psi_p^\circ(2,\Sigma)\le2^{p/2}\Psi_p^\circ(q,\Sigma)\) for \(q\ge1\).
\end{proof}

\subsection{Moment tensors of uniformly log-concave distributions}\label{subsec:ULC}

Proposition~\ref{prop:ULC-moment-tensors} bounds the Frobenius norm of the
moment tensors of a uniformly log-concave distribution by the
parity-dependent scale \(\rho_p(\Sigma)\); applied conditionally, it
controls the conditional moment tensors in
\eqref{eq:Lq-conditional-expansion}. We prove a stronger Hilbert-valued
\(L_2\) estimate for polynomials \(T(Y,\ldots,Y)\), from which the
proposition follows directly. The argument uses a bounded Stein kernel for
the uniformly log-concave random vector \(Y\), obtained from Fathi's result
\cite{fathi2019stein} after whitening. Applying the Stein
identity to the gradient of the squared norm of the polynomial yields a
recursion in the degree, which closes by induction.



We first recall the definition of Stein kernel; see \cite[Section 1]{fathi2019stein}.
\begin{definition}[Stein kernel]\label{def:stein}
Let \(Y\in\mathbb R^d\) be centered with
\(\E\|Y\|_2^2<\infty\). A measurable map
\[
\tau:\mathbb R^d\longrightarrow\mathbb R^{d\times d}
\]
with $\E\|\tau(Y)\|_{\mathsf F}^2<\infty$
is called a \emph{Stein kernel} for \(Y\) if
\[
\E\langle Y,\varphi(Y)\rangle
=
\E\left\langle
\tau(Y),\nabla\varphi(Y)
\right\rangle_{\mathsf F}
\]
for every smooth vector field
\(\varphi:\mathbb R^d\to\mathbb R^d\) satisfying $\E\left[
\|\varphi(Y)\|_2^2
+
\|\nabla\varphi(Y)\|_{\mathsf F}^2
\right]
<\infty$,
where
\[
\nabla\varphi
=
\left(
\frac{\partial\varphi_i}{\partial y_j}
\right)_{i,j=1}^d.
\]
\end{definition}

\begin{proof}[Proof of Proposition~\ref{prop:ULC-moment-tensors}]

\medskip
\noindent\textbf{Step 1: A bounded Stein kernel for $Y$.}
\medskip

We first construct a bounded Stein kernel for \(Y\). Set
\[
\widetilde Y:=\Sigma^{-1/2}Y.
\]
The density of \(\widetilde Y\) is proportional to
\(\exp(-\widetilde V)\), where
\[
\widetilde V(y):=V(\Sigma^{1/2}y).
\]
Moreover,
\begin{equation*}
\nabla^2\widetilde V(y)
=
\Sigma^{1/2}
\nabla^2V(\Sigma^{1/2}y)
\Sigma^{1/2}
\succeq
I_d.
\end{equation*}
Hence Fathi's bounded Stein-kernel result
\cite[Corollary~2.4]{fathi2019stein} yields a measurable symmetric
positive semi-definite Stein kernel \(\widetilde\tau\) for
\(\widetilde Y\) such that
\begin{equation}
0
\preceq
\widetilde\tau(\widetilde Y)
\preceq
I_d
\label{eq:ULC-whitened-Stein-order}
\end{equation}
almost surely.

Define
\begin{equation}
\tau(y)
:=
\Sigma^{1/2}
\widetilde\tau(\Sigma^{-1/2}y)
\Sigma^{1/2}.
\label{eq:ULC-Stein-transform}
\end{equation}
Then \(\tau\) is a Stein kernel for \(Y\). Indeed, let \(\varphi:\mathbb R^d\to\mathbb R^d\) be a smooth vector
field satisfying the integrability condition in Definition \ref{def:stein},
and define
\[
\psi(y)
:=
\Sigma^{1/2}\varphi(\Sigma^{1/2}y).
\]
Since the change of variables is linear, \(\psi\) satisfies the
corresponding integrability condition under the law of
\(\widetilde Y\).
Since
\[
\nabla\psi(y)
=
\Sigma^{1/2}
\nabla\varphi(\Sigma^{1/2}y)
\Sigma^{1/2},
\]
the Stein identity for \(\widetilde Y\) gives
\begin{align*}
\E\langle Y,\varphi(Y)\rangle
&=
\E
\left\langle
\widetilde Y,
\Sigma^{1/2}\varphi(\Sigma^{1/2}\widetilde Y)
\right\rangle\\
&=
\E
\left\langle
\widetilde\tau(\widetilde Y),
\Sigma^{1/2}
\nabla\varphi(\Sigma^{1/2}\widetilde Y)
\Sigma^{1/2}
\right\rangle_{\mathsf F}\\
&=
\E
\left\langle
\tau(Y),\nabla\varphi(Y)
\right\rangle_{\mathsf F}.
\end{align*}
Thus,
\begin{equation}
\E\langle Y,\varphi(Y)\rangle
=
\E
\left\langle
\tau(Y),\nabla\varphi(Y)
\right\rangle_{\mathsf F}.
\label{eq:ULC-Stein}
\end{equation}
Furthermore, \eqref{eq:ULC-whitened-Stein-order} and
\eqref{eq:ULC-Stein-transform} imply
\begin{equation}
0
\preceq
\tau(Y)
\preceq
\Sigma
\label{eq:ULC-Stein-order}
\end{equation}
almost surely. Consequently, almost surely,
\begin{equation}
\|\tau(Y)\|
\le
\|\Sigma\|,
\qquad
\|\tau(Y)\|_{\mathsf F}
\le
\|\Sigma\|_{\mathsf F}.
\label{eq:ULC-Stein-norms}
\end{equation}

Finally, the uniform convexity assumption implies that \(Y\) has moments
of every order. Indeed,
\[
V(y)
\ge
V(0)+\langle\nabla V(0),y\rangle
+\frac12\langle y,\Sigma^{-1}y\rangle,
\]
so the density of \(Y\) is bounded by a constant multiple of a shifted
Gaussian density. Hence every polynomial vector field, together with its
Jacobian, satisfies the integrability condition in Definition \ref{def:stein}. Therefore, \eqref{eq:ULC-Stein} applies to all polynomial
vector fields used below.

\medskip
\noindent\textbf{Step 2: A Hilbert-valued \(L_2\) estimate.}
\medskip

We prove the following stronger result. Set $\rho_0(\Sigma):=1$.
For every integer \(p\ge0\), every finite-dimensional real Hilbert space
\(\mathcal H\), and every \(p\)-linear map
\[
T:(\mathbb R^d)^p\longrightarrow\mathcal H,
\]
one has
\begin{equation}\label{eq:ULC-Hilbert-bound}
\left(
\E
\big\|
T(Y,\ldots,Y)
\big\|_{\mathcal H}^2
\right)^{1/2}
\le
C_p
\rho_p(\Sigma)
\|T\|_{\mathsf F},
\end{equation}
where \(C_p\) depends only on \(p\). For \(p=0\), a \(0\)-linear map is understood as a fixed vector in
\(\mathcal H\), and its Frobenius norm is its Hilbert-space norm.

Since symmetrizing \(T\) does not change \(T(Y,\ldots,Y)\) and cannot
increase \(\|T\|_{\mathsf F}\), we may assume without loss of generality
that \(T\) is symmetric. The case \(p=0\) is immediate with \(C_0=1\). For \(p=1\), applying
\eqref{eq:ULC-Stein} to linear vector fields gives
\[
\Cov(Y)
=
\E\tau(Y).
\]
Therefore, \eqref{eq:ULC-Stein-order} implies
\[
\Cov(Y)
\preceq
\Sigma.
\]
For every linear map \(T:\mathbb R^d\to\mathcal H\), it follows that
\[
\E\|T(Y)\|_{\mathcal H}^2
=
\tr
\left(
T\Cov(Y)T^*
\right)
\le \tr(T\Sigma T^*) \le
\|\Sigma\|
\tr(TT^*)
=
\|\Sigma\|\|T\|_{\mathsf F}^2.
\]
Since $\rho_1(\Sigma)=\|\Sigma\|^{1/2}$, then \eqref{eq:ULC-Hilbert-bound} holds for \(p=1\) with \(C_1=1\).

Fix \(p\ge2\), and suppose that
\eqref{eq:ULC-Hilbert-bound} holds through order \(p-1\) for every
finite-dimensional Hilbert space. Set
\[
P_T(y):=T(y,\ldots,y),
\qquad
f_T(y):=\|P_T(y)\|_{\mathcal H}^2,
\]
and write 
\[
m_p
:=
\left(
\E\|P_T(Y)\|_{\mathcal H}^2
\right)^{1/2}.
\]
Since \(f_T\) is homogeneous of degree \(2p\), Euler's identity gives
\[
\langle y,\nabla f_T(y)\rangle
=
2p f_T(y).
\]
Therefore,
\[
\E\langle Y,\nabla f_T(Y)\rangle
=
2p\,\E f_T(Y)
=
2p m_p^2.
\]
Applying the Stein identity \eqref{eq:ULC-Stein} to the vector field $\varphi(y):=\nabla f_T(y)$, whose Jacobian is $\nabla\varphi(y)=\nabla^2f_T(y)$,
we obtain
\[
\E\langle Y,\nabla f_T(Y)\rangle
=
\E
\left\langle
\tau(Y),\nabla^2f_T(Y)
\right\rangle_{\mathsf F}.
\]
Combining the two identities yields
\begin{equation}
2p m_p^2
=
\E
\left\langle
\tau(Y),\nabla^2f_T(Y)
\right\rangle_{\mathsf F}.
\label{eq:ULC-Stein-polynomial}
\end{equation}

A direct differentiation gives
\begin{align}
\frac{\partial^2f_T}{\partial y_i\partial y_j}(y)
=
2p^2
\left\langle
T(y,\ldots,y,e_i),
T(y,\ldots,y,e_j)
\right\rangle_{\mathcal H}
+
2p(p-1)
\left\langle
P_T(y),
T(y,\ldots,y,e_i,e_j)
\right\rangle_{\mathcal H}.
\label{eq:ULC-Hessian}
\end{align}
Substituting \eqref{eq:ULC-Hessian} into
\eqref{eq:ULC-Stein-polynomial} and dividing by \(2p\), we obtain
\begin{align}\label{eq:ULC_aux1}
m_p^2
&=
\underbrace{p
\E
\sum_{i,j=1}^d
\tau_{ij}(Y)
\left\langle
T(Y,\ldots,Y,e_i),
T(Y,\ldots,Y,e_j)
\right\rangle_{\mathcal H}}_{=:\circnum{1}}\nonumber\\
&\quad +
\underbrace{(p-1)
\E
\left\langle
P_T(Y),
\sum_{i,j=1}^d
\tau_{ij}(Y)
T(Y,\ldots,Y,e_i,e_j)
\right\rangle_{\mathcal H}}_{=:\circnum{2}}.
\end{align}

Define the symmetric \((p-1)\)-linear map $\mathcal D_1T:
(\mathbb R^d)^{p-1}
\longrightarrow
\mathcal H^d$
by
\[
(\mathcal D_1T)(y_1,\ldots,y_{p-1})
:=
\left(
T(y_1,\ldots,y_{p-1},e_j)
\right)_{j=1}^d,
\]
where \(\mathcal H^d\) is equipped with the product Hilbert-space norm
\[
\|(h_j)_{j=1}^d\|_{\mathcal H^d}^2
:=
\sum_{j=1}^d\|h_j\|_{\mathcal H}^2.
\]
Similarly, define the symmetric \((p-2)\)-linear map $\mathcal D_2T:
(\mathbb R^d)^{p-2}
\longrightarrow
\mathcal H^{d\times d}$
by
\[
(\mathcal D_2T)(y_1,\ldots,y_{p-2})
:=
\left(
T(y_1,\ldots,y_{p-2},e_i,e_j)
\right)_{i,j=1}^d,
\]
where
\[
\|(h_{ij})_{i,j=1}^d\|_{\mathcal H^{d\times d}}^2
:=
\sum_{i,j=1}^d\|h_{ij}\|_{\mathcal H}^2.
\]
By the definition of the Frobenius norm of a multilinear map,
\begin{align*}
\|\mathcal D_1T\|_{\mathsf F}^2
&=
\sum_{i_1,\ldots,i_{p-1}=1}^d
\big\|
(\mathcal D_1T)(e_{i_1},\ldots,e_{i_{p-1}})
\big\|_{\mathcal H^d}^2\\
&=
\sum_{i_1,\ldots,i_{p-1},j=1}^d
\big\|
T(e_{i_1},\ldots,e_{i_{p-1}},e_j)
\big\|_{\mathcal H}^2
=
\|T\|_{\mathsf F}^2.
\end{align*}
Similarly,
\begin{align*}
\|\mathcal D_2T\|_{\mathsf F}^2
&=
\sum_{i_1,\ldots,i_{p-2}=1}^d
\big\|
(\mathcal D_2T)(e_{i_1},\ldots,e_{i_{p-2}})
\big\|_{\mathcal H^{d\times d}}^2\\
&=
\sum_{i_1,\ldots,i_{p-2},i,j=1}^d
\big\|
T(e_{i_1},\ldots,e_{i_{p-2}},e_i,e_j)
\big\|_{\mathcal H}^2
=
\|T\|_{\mathsf F}^2.
\end{align*}
Consequently,
\[
\|\mathcal D_1T\|_{\mathsf F}
=
\|\mathcal D_2T\|_{\mathsf F}
=
\|T\|_{\mathsf F}.
\]
The induction hypothesis therefore gives
\begin{align}
\bigg(
\E
\sum_{j=1}^d
\big\|
T(Y,\ldots,Y,e_j)
\big\|_{\mathcal H}^2
\bigg)^{1/2}
&\le
C_{p-1}
\rho_{p-1}(\Sigma)
\|T\|_{\mathsf F},
\label{eq:ULC-first-derivative}\\
\bigg(
\E
\sum_{i,j=1}^d
\big\|
T(Y,\ldots,Y,e_i,e_j)
\big\|_{\mathcal H}^2
\bigg)^{1/2}
&\le
C_{p-2}
\rho_{p-2}(\Sigma)
\|T\|_{\mathsf F}.
\label{eq:ULC-second-derivative}
\end{align}

For every fixed \(y\in\mathbb R^d\), the matrix
\[
\left(
\left\langle
T(y,\ldots,y,e_i),
T(y,\ldots,y,e_j)
\right\rangle_{\mathcal H}
\right)_{i,j=1}^d
\]
is positive semi-definite. Hence
\eqref{eq:ULC-Stein-order} implies
\begin{align}
\sum_{i,j=1}^d
\tau_{ij}(Y)
\left\langle
T(Y,\ldots,Y,e_i),
T(Y,\ldots,Y,e_j)
\right\rangle_{\mathcal H}
\le
\|\Sigma\|
\sum_{j=1}^d
\big\|
T(Y,\ldots,Y,e_j)
\big\|_{\mathcal H}^2.
\label{eq:ULC-first-Hessian-bound}
\end{align}
By \eqref{eq:ULC-first-Hessian-bound} and
\eqref{eq:ULC-first-derivative}, the term $\circnum{1}$ in \eqref{eq:ULC_aux1} is bounded by
\begin{align*}
\circnum{1}
\le
p\|\Sigma\|
\E
\sum_{j=1}^d
\big\|
T(Y,\ldots,Y,e_j)
\big\|_{\mathcal H}^2\le
pC_{p-1}^2
\|\Sigma\|
\rho_{p-1}(\Sigma)^2
\|T\|_{\mathsf F}^2.
\end{align*}

Moreover, the Cauchy--Schwarz inequality and
\eqref{eq:ULC-Stein-norms} give
\begin{align}
\bigg\|
\sum_{i,j=1}^d
\tau_{ij}(Y)
T(Y,\ldots,Y,e_i,e_j)
\bigg\|_{\mathcal H}
&\le
\|\tau(Y)\|_{\mathsf F}
\bigg(
\sum_{i,j=1}^d
\big\|
T(Y,\ldots,Y,e_i,e_j)
\big\|_{\mathcal H}^2
\bigg)^{1/2}
\nonumber\\
&\le
\|\Sigma\|_{\mathsf F}
\bigg(
\sum_{i,j=1}^d
\big\|
T(Y,\ldots,Y,e_i,e_j)
\big\|_{\mathcal H}^2
\bigg)^{1/2}.
\label{eq:ULC-second-Hessian-bound}
\end{align}
Applying the Cauchy--Schwarz inequality and using \eqref{eq:ULC-second-Hessian-bound}, we obtain
\begin{align*}
\circnum{2}&\le (p-1)\left|
\E
\left\langle
P_T(Y),
\sum_{i,j=1}^d
\tau_{ij}(Y)
T(Y,\ldots,Y,e_i,e_j)
\right\rangle_{\mathcal H}
\right|\\
&\le (p-1)
\|\Sigma\|_{\mathsf F}
\E
\left[
\|P_T(Y)\|_{\mathcal H}
\bigg(
\sum_{i,j=1}^d
\big\|
T(Y,\ldots,Y,e_i,e_j)
\big\|_{\mathcal H}^2
\bigg)^{1/2}
\right]\\
&\le (p-1)
\|\Sigma\|_{\mathsf F}
\left(
\E\|P_T(Y)\|_{\mathcal H}^2
\right)^{1/2}
\bigg(
\E
\sum_{i,j=1}^d
\big\|
T(Y,\ldots,Y,e_i,e_j)
\big\|_{\mathcal H}^2
\bigg)^{1/2}\\
&\le (p-1)
C_{p-2}
\|\Sigma\|_{\mathsf F}
\rho_{p-2}(\Sigma)
m_p
\|T\|_{\mathsf F},
\end{align*}
where the last step uses \eqref{eq:ULC-second-derivative} and the definition of \(m_p\). 

Combining the estimates of $\circnum{1}$, $\circnum{2}$, and \eqref{eq:ULC_aux1} yields
\begin{align}
m_p^2
\le
pC_{p-1}^2
\|\Sigma\|
\rho_{p-1}(\Sigma)^2
\|T\|_{\mathsf F}^2
+
(p-1)C_{p-2}
\|\Sigma\|_{\mathsf F}
\rho_{p-2}(\Sigma)
m_p
\|T\|_{\mathsf F}.
\label{eq:ULC-moment-recursion}
\end{align}
The definition of \(\rho_p(\Sigma)\), together with
\(\|\Sigma\|\le\|\Sigma\|_{\mathsf F}\), implies
\[
\|\Sigma\|^{1/2}\rho_{p-1}(\Sigma)
\le
\rho_p(\Sigma),\qquad
\|\Sigma\|_{\mathsf F}\rho_{p-2}(\Sigma)
=
\rho_p(\Sigma).
\]
Thus \eqref{eq:ULC-moment-recursion} yields
\[
m_p^2
\le
pC_{p-1}^2
\rho_p(\Sigma)^2
\|T\|_{\mathsf F}^2
+
(p-1)C_{p-2}
\rho_p(\Sigma)
m_p
\|T\|_{\mathsf F}.
\]
Since $t^2\le a^2+bt
$ implies $t\le a+b$ for $t,a,b\ge 0$,
we conclude that
\[
m_p
\le
\left(
\sqrt p\,C_{p-1}
+
(p-1)C_{p-2}
\right)
\rho_p(\Sigma)
\|T\|_{\mathsf F}.
\]
The induction therefore closes upon setting
\[
C_p
:=
\sqrt p\,C_{p-1}
+
(p-1)C_{p-2}.
\]

\medskip
\noindent\textbf{Step 3: Completion of the proof.}
\medskip

By Frobenius duality, Cauchy--Schwarz, and \eqref{eq:ULC-Hilbert-bound} in \textbf{Step~2},
\begin{align*}
\bigl\|\E Y^{\otimes p}\bigr\|_{\mathsf F}
=
\sup_{\|A\|_{\mathsf F}=1}
\left|
\E\left\langle A,Y^{\otimes p}\right\rangle_{\mathsf F}
\right| \le
\sup_{\|A\|_{\mathsf F}=1}
\left(
\E\left|
\left\langle A,Y^{\otimes p}\right\rangle_{\mathsf F}
\right|^2
\right)^{1/2}
\lesssim_p \rho_p(\Sigma).
\end{align*}
This completes the proof.
\end{proof}

\section{Auxiliary proofs}\label{app:proofs}

\subsection{Proof of Lemma \ref{lem:Hilbert-moment-reduction}}\label{subsec:Hilbert-moment-reduction}

\begin{proof}[Proof of Lemma \ref{lem:Hilbert-moment-reduction}]
Let
\[
S:=\sum_{i=1}^N Z_i,\qquad S':=\sum_{i=1}^N Z_i',
\]
where
\(Z_1',\ldots,Z_N'\) is an independent copy of
\(Z_1,\ldots,Z_N\). Expanding the squared Hilbert-space norm gives
\[
\|S\|_{\mathcal H}^2
=
\sum_{i=1}^N\|Z_i\|_{\mathcal H}^2
+
\sum_{i\ne j}
\langle Z_i,Z_j\rangle_{\mathcal H}.
\]
Therefore,
\begin{align*}
\|S\|_{L_q(\mathcal H)}^2 = \left\|\|S\|_{\mathcal{H}}^2\right\|_{L_{q/2}}
&\le
\bigg\|
\sum_{i=1}^N\|Z_i\|_{\mathcal H}^2
\bigg\|_{L_{q/2}}
+
\bigg\|
\sum_{i\ne j}
\langle Z_i,Z_j\rangle_{\mathcal H}
\bigg\|_{L_{q/2}}\\
&=
\left\|
\bigg(
\sum_{i=1}^N\|Z_i\|_{\mathcal H}^2
\bigg)^{1/2}
\right\|_{L_q}^2
+
\bigg\|
\sum_{i\ne j}
\langle Z_i,Z_j\rangle_{\mathcal H}
\bigg\|_{L_{q/2}}.
\end{align*}
Applying the Hilbert-space decoupling inequality
\cite[Exercise~6.1.4]{vershynin2018high} with
\(a_{ij}=1\) for all \(i,j\) and the convex function
\(F(t)=|t|^{q/2}\), and then taking both sides to the power \(2/q\),
we obtain
\[
\bigg\|
\sum_{i\ne j}
\langle Z_i,Z_j\rangle_{\mathcal H}
\bigg\|_{L_{q/2}}
\le
4\,
\bigg\|
\sum_{i,j=1}^N
\langle Z_i,Z_j'\rangle_{\mathcal H}
\bigg\|_{L_{q/2}}
=
4\,\|\langle S,S'\rangle_{\mathcal H}\|_{L_{q/2}}.
\]
On the event \(\{S\ne0\}\), conditioning on \(S\) and using the
independence of \(S\) and \(S'\), we have
\begin{align*}
\left(
\E\Big[
|\langle S,S'\rangle_{\mathcal H}|^{q/2}
\,\big| \,
S
\Big]
\right)^{2/q}
&=
\|S\|_{\mathcal H}
\left(
\E\left[
\left|
\bigg\langle
\frac{S}{\|S\|_{\mathcal H}},S'
\bigg\rangle_{\mathcal H}
\right|^{q/2}
\,\middle|\,
S
\right]
\right)^{2/q}\\
&\le
\|S\|_{\mathcal H}
\sup_{\|h\|_{\mathcal H}=1}
\|\langle h,S'\rangle_{\mathcal H}\|_{L_{q/2}}.
\end{align*}
The same inequality is trivial on \(\{S=0\}\). Consequently,
\begin{align*}
\|\langle S,S'\rangle_{\mathcal H}\|_{L_{q/2}}
&\le
\sup_{\|h\|_{\mathcal H}=1}
\bigg\|
\sum_{i=1}^N
\langle h,Z_i'\rangle_{\mathcal H}
\bigg\|_{L_{q/2}}
\|S\|_{L_{q/2}(\mathcal H)}\\
&=
\sup_{\|h\|_{\mathcal H}=1}
\bigg\|
\sum_{i=1}^N
\langle h,Z_i\rangle_{\mathcal H}
\bigg\|_{L_{q/2}}
\|S\|_{L_{q/2}(\mathcal H)}\\
&\le
\sup_{\|h\|_{\mathcal H}=1}
\bigg\|
\sum_{i=1}^N
\langle h,Z_i\rangle_{\mathcal H}
\bigg\|_{L_{q/2}}
\|S\|_{L_q(\mathcal H)}.
\end{align*}
It follows that
\begin{align*}
\|S\|_{L_q(\mathcal H)}^2
\le
\left\|
\bigg(
\sum_{i=1}^N \|Z_i\|_{\mathcal H}^2
\bigg)^{1/2}
\right\|_{L_q}^2+
4
\sup_{\|h\|_{\mathcal H}=1}
\left\|
\sum_{i=1}^N
\langle h,Z_i\rangle_{\mathcal H}
\right\|_{L_{q/2}}
\|S\|_{L_q(\mathcal H)}.
\end{align*}
Since \(t^2\le a^2+4bt\) implies \(t\le a+4b\) for $t,a,b\ge 0$, the upper bound
follows.

For the reverse inequality, let
\(\varepsilon_1,\ldots,\varepsilon_N\) be independent Rademacher random
variables, independent of everything else. Since
\[
\E\left[
(Z_i-Z_i')_{i=1}^N
\,\middle|\,
Z_1,\ldots,Z_N
\right]
=
(Z_i)_{i=1}^N,
\]
conditional Jensen's inequality in \(\ell_2^N(\mathcal H)\) gives
\[
\left\|
\bigg(
\sum_{i=1}^N\|Z_i\|_{\mathcal H}^2
\bigg)^{1/2}
\right\|_{L_q}
\le
\left(
\E
\bigg(
\sum_{i=1}^N\|Z_i-Z_i'\|_{\mathcal H}^2
\bigg)^{q/2}
\right)^{1/q}.
\]
For every fixed realization of \((Z_i,Z_i')_{i=1}^N\), Rademacher
orthogonality and \(q\ge2\) imply
\begin{align*}
\bigg(
\sum_{i=1}^N\|Z_i-Z_i'\|_{\mathcal H}^2
\bigg)^{q/2}
=
\left(
\E_\varepsilon
\bigg\|
\sum_{i=1}^N\varepsilon_i(Z_i-Z_i')
\bigg\|_{\mathcal H}^2
\right)^{q/2}\le
\E_\varepsilon
\bigg\|
\sum_{i=1}^N\varepsilon_i(Z_i-Z_i')
\bigg\|_{\mathcal H}^q.
\end{align*}
Consequently,
\begin{align*}
\left\|
\bigg(
\sum_{i=1}^N\|Z_i\|_{\mathcal H}^2
\bigg)^{1/2}
\right\|_{L_q}
&\le
\left(
\E\,\E_\varepsilon
\bigg\|
\sum_{i=1}^N\varepsilon_i(Z_i-Z_i')
\bigg\|_{\mathcal H}^q
\right)^{1/q}=
\|S-S'\|_{L_q(\mathcal H)}
\le
2\|S\|_{L_q(\mathcal H)}.
\end{align*}
Here the equality follows because the random vectors
\(Z_i-Z_i'\) are independent and symmetric.

Moreover, for every \(h\in\mathcal H\) with \(\|h\|_{\mathcal H}=1\),
\[
\bigg\|
\sum_{i=1}^N\langle h,Z_i\rangle_{\mathcal H}
\bigg\|_{L_{q/2}}
=
\|\langle h,S\rangle_{\mathcal H}\|_{L_{q/2}}
\le
\|\langle h,S\rangle_{\mathcal H}\|_{L_q}
\le
\|S\|_{L_q(\mathcal H)}.
\]
Therefore,
\[
\left\|
\bigg(
\sum_{i=1}^N\|Z_i\|_{\mathcal H}^2
\bigg)^{1/2}
\right\|_{L_q}
+
\sup_{\|h\|_{\mathcal H}=1}
\bigg\|
\sum_{i=1}^N\langle h,Z_i\rangle_{\mathcal H}
\bigg\|_{L_{q/2}}
\le
3\|S\|_{L_q(\mathcal H)}.
\]
Together with the upper bound, this completes the proof.

\end{proof}

\subsection{Proof of Lemma~\ref{lem:subgaussian-directional-tensor-moments}}\label{subsec:subgaussian-directional-tensor-moments}

\begin{proof}[Proof of Lemma~\ref{lem:subgaussian-directional-tensor-moments}]

\medskip
\noindent\textbf{Sub-Gaussian upper bound.}
\medskip

Set \(E:=\operatorname{Ran}(\Sigma)\). Since \(\E\langle v,X\rangle^2=0\)
for \(v\in\ker(\Sigma)\), we have \(X\in E\) almost surely. Let
\(\Sigma_E\) denote the restriction of \(\Sigma\) to \(E\), and set
\(Y:=K^{-1}\Sigma_E^{-1/2}X\). By the definition of the sub-Gaussian
constant \(K\), after a universal rescaling \(Y\) satisfies the
\(1\)-sub-Gaussian tail convention used in the comparison theorem \cite[Theorem~1.1]{vanhandel2025subgaussian}. Consequently,
\(Y\preceq_{\mathrm{cx}}CG_E\) with \(G_E\sim\mcN(0,I_E)\) for a
universal constant \(C>0\), and since convex order is preserved under
linear maps, applying
\(C^{-1}\Sigma_E^{1/2}\) to both sides gives
\[
(CK)^{-1}X\preceq_{\mathrm{cx}}\Sigma_E^{1/2}G_E=G_\Sigma.
\]

For \(A\in(\mathbb R^d)^{\otimes p}\), define the scalar-valued
\(p\)-linear map
\(T_A(x_1,\ldots,x_p):=\langle A,x_1\otimes\cdots\otimes x_p\rangle_{\mathsf F}\),
whose Frobenius norm is \(\|T_A\|_{\mathsf F}=\|A\|_{\mathsf F}\).
Applying Theorem~\ref{thm:main2} to \((CK)^{-1}X\) and \(T_A\), and using
the homogeneity of \(T_A\), gives
\begin{align*}
\left\|
\left\langle A,X^{\otimes p}\right\rangle_{\mathsf F}
\right\|_{L_q}
=
(CK)^p
\left\|
T_A\bigl((CK)^{-1}X,\ldots,(CK)^{-1}X\bigr)
\right\|_{L_q}
\lesssim_p
K^p\Psi_p(q,\Sigma)\|A\|_{\mathsf F}.
\end{align*}
Since \(\|V-\E V\|_{L_q}\le2\|V\|_{L_q}\) for every scalar
\(V\in L_q\), it follows that
\begin{align*}
\sup_{\|A\|_{\mathsf F}=1}
\left\|
\left\langle A,X^{\otimes p}-\E X^{\otimes p}\right\rangle_{\mathsf F}
\right\|_{L_q}
&\lesssim_p
K^p\Psi_p(q,\Sigma)\\
&\asymp_p K^p
\begin{cases}
\|\Sigma\|_{\mathsf F}^{p/2}+\|\Sigma\|^{p/2}q^{p/2},
& p \text{ even},\\[3pt]
\|\Sigma\|_{\mathsf F}^{(p-1)/2}\|\Sigma\|^{1/2}\sqrt{q}+\|\Sigma\|^{p/2}q^{p/2},
& p \text{ odd},
\end{cases}
\end{align*}
where the final comparison follows since the terms defining
\(\Psi_p(q,\Sigma)\) form a geometric progression.

 \medskip
\noindent\textbf{Gaussian upper bounds.}
\medskip

Now let \(X\sim\mcN(0,\Sigma)\). Applying 
Proposition~\ref{prop:Hilbert-Gaussian-tensor-Lq} with
\(\mathcal H=\mathbb R\) and \(T=T_A\), we obtain
\[
\left\|
\left\langle A,X^{\otimes p}\right\rangle_{\mathsf F}
\right\|_{L_q}
\lesssim_p
\Psi_p(q,\Sigma)\|A\|_{\mathsf F},
\]
and
\[
\left\|
\left\langle A,X^{\otimes p}-\E X^{\otimes p}\right\rangle_{\mathsf F}
\right\|_{L_q}
\lesssim_p
\Psi_p^\circ(q,\Sigma)\|A\|_{\mathsf F},
\]
where 
\[
\Psi_p(q,\Sigma)
=
\sum_{\substack{0\le j\le p\\ j\equiv p\;(\mathrm{mod}\,2)}}
\|\Sigma\|_{\mathsf F}^{(p-j)/2}\|\Sigma\|^{j/2}q^{j/2}\asymp_p
\begin{cases}
\|\Sigma\|_{\mathsf F}^{p/2}
+
\|\Sigma\|^{p/2}q^{p/2},
& p \text{ even},\\[3pt]
\|\Sigma\|_{\mathsf F}^{(p-1)/2}\|\Sigma\|^{1/2}\sqrt q
+
\|\Sigma\|^{p/2}q^{p/2},
& p \text{ odd},
\end{cases}
\]
and
\[
\Psi_p^\circ(q,\Sigma)
=
\sum_{\substack{1\le j\le p\\ j\equiv p\;(\mathrm{mod}\,2)}}
\|\Sigma\|_{\mathsf F}^{(p-j)/2}\|\Sigma\|^{j/2}q^{j/2}\asymp_p
\begin{cases}
\|\Sigma\|_{\mathsf F}^{p/2-1}\|\Sigma\|q
+
\|\Sigma\|^{p/2}q^{p/2},
& p \text{ even},\\[3pt]
\|\Sigma\|_{\mathsf F}^{(p-1)/2}\|\Sigma\|^{1/2}\sqrt q
+
\|\Sigma\|^{p/2}q^{p/2},
& p \text{ odd}.
\end{cases}
\]
Taking the supremum over \(\|A\|_{\mathsf F}=1\) proves the desired
upper bounds.

\medskip
\noindent\textbf{Gaussian lower bounds.} 
\medskip

By orthogonal invariance, assume that
\[
\Sigma=\operatorname{diag}(\lambda_1,\ldots,\lambda_d),
\qquad
\lambda_1=\|\Sigma\|,
\qquad
X=\big(\lambda_1^{1/2}\xi_1,\ldots, \lambda_d^{1/2}\xi_d\big),
\]
where \(\xi_1,\ldots,\xi_d\) are independent standard Gaussian random
variables. 

 We first prove the lower bound for the centered version. 

Let \(v=e_1\). Testing against \(A=v^{\otimes p}\) gives
\[
\left\langle
A,X^{\otimes p}-\E X^{\otimes p}
\right\rangle_{\mathsf F}
=
\|\Sigma\|^{p/2}
\left(
\xi_1^p-\E\xi_1^p
\right).
\]
Since
\[
\left\|
\xi_1^p-\E\xi_1^p
\right\|_{L_q}
\asymp_p
q^{p/2},
\qquad q\ge1,
\]
we obtain
\begin{equation}\label{eq:gaussian-directional-highest-chaos-lower}
\sup_{\|A\|_{\mathsf F}=1}
\left\|
\left\langle
A,X^{\otimes p}-\E X^{\otimes p}
\right\rangle_{\mathsf F}
\right\|_{L_q}
\gtrsim_p
\|\Sigma\|^{p/2}q^{p/2}.
\end{equation}

Set
\[
\sigma_{\perp}
:=
\bigg(
\sum_{j=2}^d\lambda_j^2
\bigg)^{1/2}.
\]
Suppose first that $\sigma_{\perp}
\ge
\frac12\|\Sigma\|_{\mathsf F}$.
Define
\[
B_{\perp}
:=
\frac{\operatorname{diag}(0,\lambda_2,\ldots,\lambda_d)}
{\sigma_{\perp}},
\qquad
Q_{\perp}:=X^\top B_{\perp}X.
\]
Then
\[
\|B_{\perp}\|_{\mathsf F}=1,
\qquad
Q_{\perp}
=
\frac1{\sigma_{\perp}}
\sum_{j=2}^d\lambda_j^2\xi_j^2,
\qquad
\E Q_{\perp}=\sigma_{\perp},
\]
and \(Q_{\perp}\) is independent of \(\xi_1\).

Suppose that \(p=2k\) is even. Take $A:=(v\otimes v)\otimes B_{\perp}^{\otimes(k-1)}$,
with the convention \(B_{\perp}^{\otimes0}=1\). Then
\(\|A\|_{\mathsf F}=1\), and
\[
\left\langle
A,X^{\otimes p}
\right\rangle_{\mathsf F}= \langle v,X\rangle^2 (X^{\top} B_{\perp} X)^{k-1}
=
\|\Sigma\|\xi_1^2Q_{\perp}^{k-1}.
\]
By contractivity of conditional expectation in \(L_q\) and Jensen's inequality,
\begin{align}\label{eq:last_lemma_aux1}
\left\|
\left\langle
A,X^{\otimes p}-\E X^{\otimes p}
\right\rangle_{\mathsf F}
\right\|_{L_q}
&\ge
\left\|
\E\left[
\left\langle
A,X^{\otimes p}-\E X^{\otimes p}
\right\rangle_{\mathsf F}
\,\middle|\,
\xi_1
\right]
\right\|_{L_q}\nonumber \\
&=
\|\Sigma\|
\left(\E Q_{\perp}^{k-1}\right)
\|\xi_1^2-1\|_{L_q}\nonumber\\
&\ge
\|\Sigma\|
(\E Q_{\perp})^{k-1}
\|\xi_1^2-1\|_{L_q}\nonumber\\
&= \|\Sigma\|
(\sigma_{\perp})^{k-1}
\|\xi_1^2-1\|_{L_q}\nonumber\\
&\gtrsim_p
\|\Sigma\|
\|\Sigma\|_{\mathsf F}^{k-1}q.
\end{align}

Suppose instead that \(p=2k+1\) is odd. Take $A:=v\otimes B_{\perp}^{\otimes k}$. Then \(\|A\|_{\mathsf F}=1\), and
\[
\left\langle
A,X^{\otimes p}
\right\rangle_{\mathsf F}= \left\langle v,X\right\rangle (X^{\top} B_{\perp} X)^k
=
\|\Sigma\|^{1/2}\xi_1Q_{\perp}^k.
\]
This random variable is centered. By independence, Jensen's inequality,
and the standard Gaussian moment estimate, we have
\begin{align}\label{eq:last_lemma_aux2}
\left\|
\left\langle
A,X^{\otimes p}
\right\rangle_{\mathsf F}
\right\|_{L_q}
&=
\|\Sigma\|^{1/2}
\|\xi_1\|_{L_q}
\|Q_{\perp}\|_{L_{kq}}^k\nonumber\\
&\ge
\|\Sigma\|^{1/2}
\|\xi_1\|_{L_q}
(\E Q_{\perp})^k \nonumber\\
&= \|\Sigma\|^{1/2}
\|\xi_1\|_{L_q}
(\sigma_{\perp})^k  \nonumber\\
&\gtrsim_p
\|\Sigma\|^{1/2}
\|\Sigma\|_{\mathsf F}^k
\sqrt q.
\end{align}

Together with \eqref{eq:gaussian-directional-highest-chaos-lower}, the
bounds \eqref{eq:last_lemma_aux1} and \eqref{eq:last_lemma_aux2} prove the
lower bound when \(\sigma_{\perp}\ge\frac12\|\Sigma\|_{\mathsf F}\).

It remains to consider $\sigma_{\perp}
<
\frac12\|\Sigma\|_{\mathsf F}$.
Since $\|\Sigma\|_{\mathsf F}^2
=
\|\Sigma\|^2+\sigma_{\perp}^2$,
this gives $\|\Sigma\|_{\mathsf F}
\le
\frac{2}{\sqrt3}\|\Sigma\|$, and consequently
\[
\|\Sigma\|_{\mathsf F}^{p/2-1}\|\Sigma\|q
\lesssim_p\|\Sigma\|^{p/2}q^{p/2},
\qquad
\|\Sigma\|_{\mathsf F}^{(p-1)/2}\|\Sigma\|^{1/2}\sqrt q
\lesssim_p\|\Sigma\|^{p/2}q^{p/2}.
\]
Therefore \eqref{eq:gaussian-directional-highest-chaos-lower} already
dominates the intermediate scale in this case, which completes the proof
of the centered lower bound.

\medskip

We now turn to the non-centered version. Since
\(\|V-\E V\|_{L_q}\le2\|V\|_{L_q}\) for every scalar \(V\in L_q\), the
centered lower bound just established gives
\begin{equation}\label{eq:last_lemma_aux3}
\sup_{\|A\|_{\mathsf F}=1}
\left\|
\left\langle A,X^{\otimes p}\right\rangle_{\mathsf F}
\right\|_{L_q}
\gtrsim_p
\begin{cases}
\|\Sigma\|_{\mathsf F}^{p/2-1}\|\Sigma\|q
+
\|\Sigma\|^{p/2}q^{p/2},
& p \text{ even},\\[3pt]
\|\Sigma\|_{\mathsf F}^{(p-1)/2}\|\Sigma\|^{1/2}\sqrt q
+
\|\Sigma\|^{p/2}q^{p/2},
& p \text{ odd}.
\end{cases}
\end{equation}
When \(p\) is odd, this is the desired bound. Suppose then that
\(p=2k\) is even, and set
\[
B:=\frac{\Sigma}{\|\Sigma\|_{\mathsf F}},
\qquad
A:=B^{\otimes k},
\]
so that \(\|A\|_{\mathsf F}=\|B\|_{\mathsf F}^k=1\) and
\[
\left\langle A,X^{\otimes p}\right\rangle_{\mathsf F}
=
\left\langle B^{\otimes k},X^{\otimes 2k}\right\rangle_{\mathsf F}
=
\left(X^\top B X\right)^k
=
\bigg(
\frac1{\|\Sigma\|_{\mathsf F}}
\sum_{j=1}^d\lambda_j^2\xi_j^2
\bigg)^{k}
\ge0.
\]
Since \(X^\top B X\ge0\) with
\(\E X^\top B X=\|\Sigma\|_{\mathsf F}\), monotonicity of
\(L_q\)-norms and Jensen's inequality give
\begin{equation}\label{eq:last_lemma_aux4}
\left\|
\left\langle A,X^{\otimes p}\right\rangle_{\mathsf F}
\right\|_{L_q}
\ge
\E\left(X^\top B X\right)^k
\ge
\left(\E X^\top B X\right)^k
=
\|\Sigma\|_{\mathsf F}^{p/2}.
\end{equation}
Combining \eqref{eq:last_lemma_aux3} and \eqref{eq:last_lemma_aux4}, and
noting that
\(\|\Sigma\|_{\mathsf F}^{p/2-1}\|\Sigma\|q
\lesssim_p\|\Sigma\|_{\mathsf F}^{p/2}+\|\Sigma\|^{p/2}q^{p/2}\),
we obtain, for \(p\) even,
\[
\sup_{\|A\|_{\mathsf F}=1}
\left\|
\left\langle A,X^{\otimes p}\right\rangle_{\mathsf F}
\right\|_{L_q}
\gtrsim_p
\|\Sigma\|_{\mathsf F}^{p/2}+\|\Sigma\|^{p/2}q^{p/2}.
\]
This completes the proof.
\end{proof}

\section*{Acknowledgments}
The authors were partly funded by the NSF CAREER award DMS-2237628.

\bibliographystyle{siam} 
\bibliography{references}

\end{document}